\documentclass[11pt,reqno]{article}
\usepackage{amsmath,amsfonts,amsthm}
\usepackage[margin=1in]{geometry}
\usepackage[normalem]{ulem}
\usepackage{amssymb}
\usepackage{appendix}

\usepackage[hyphens]{url}
\usepackage{hyperref}
\usepackage{graphicx}
\usepackage{epstopdf}
\usepackage{float}
\usepackage{subfigure}
\usepackage{verbatim}
\usepackage{enumitem}
\usepackage{color}
\usepackage{bm}
\usepackage{dsfont}

\usepackage{soul}

\numberwithin{equation}{section}

\theoremstyle{plain}

\newcommand{\lam}{\lambda}
\newcommand{\tV}{\widetilde{V}}

\newcommand{\red}{\color{black}}

\newtheorem{lemma}{Lemma}[section]
\newtheorem{theorem}{Theorem}[section]

\newtheorem{proposition}{Proposition}[section]
\newtheorem{corollary}{Corollary}[section]

\newtheorem{remark}{Remark}[section]
\allowdisplaybreaks[1]

\def \E{\mathbb{E}}

\makeatletter
\def\@setcopyright{}
\def\serieslogo@{}
\makeatother

\title{Optimal consumption and life insurance under shortfall aversion and a drawdown constraint}
\author{Xun LI\thanks{Department of Applied Mathematics, The Hong Kong Polytechnic University, Hung Hom, Kowloon, Hong Kong. Email:\texttt{li.xun@polyu.edu.hk}}
\and
Xiang YU\thanks{Department of Applied Mathematics, The Hong Kong Polytechnic University, Hung Hom, Kowloon, Hong Kong. Email:\texttt{xiang.yu@polyu.edu.hk}}
\and
Qinyi ZHANG\thanks{Department of Applied Mathematics, The Hong Kong Polytechnic University, Hung Hom, Kowloon, Hong Kong. Email:\texttt{qinyi-qy.zhang@connect.polyu.hk}}
}

\begin{document}

\date{\vspace{-0.5cm}}
\maketitle

\begin{abstract}

This paper studies a life-cycle optimal portfolio-consumption problem when the consumption performance is measured by a shortfall aversion preference under an additional drawdown constraint on consumption rate. Meanwhile, the agent also dynamically chooses her life insurance premium to maximize the expected bequest at the death time.
By using dynamic programming arguments and the dual transform, we solve the HJB variational inequality explicitly in a piecewise form across different regions and derive some thresholds of the wealth variable for the piecewise optimal feedback controls. Taking advantage of our analytical results, we are able to numerically illustrate some quantitative impacts on optimal consumption and life insurance by model parameters and discuss their financial implications.
\end{abstract}
\section{Introduction}\label{sec: intro}
Since the seminal studies \cite{Mert1969RES} and \cite{Mert1971JET} by Merton, the continuous time optimal investment and consumption problem under utility maximization has been extensively investigated by incorporating various stochastic factors, market incompleteness, trading constraints and the combination with other financial or actuarial decisions. Giving a complete literature review is beyond the scope of this paper. Among different model generalizations, the optimal life insurance in the context of utility maximization has attracted considerable attention. Richard \cite{Richard1975JFE} proposes the optimal dynamic life insurance problem for the first time by combining the portfolio and consumption control under a given distribution of a bounded death time. Pliska and Ye \cite{PliskaYe2007JBF} further study a similar optimal life insurance and consumption problem for an income earner when the lifetime random variable is unbounded. Later, Ye \cite{Ye2007ACC} extends the model in \cite{PliskaYe2007JBF} by considering the dynamic portfolio in a risky asset. Huang and Milevsky \cite{HuangMilevsky2008JBF} solve a portfolio choice problem that includes mortality-contingent claims and labor income under general HARA utilities.
Duarte et al. \cite{DuartePinheiroPP2011DGS1} extend \cite{Ye2007ACC} to allow for multiple risky assets. Recently, wei et al. \cite{WeiChengJW2020IME} solve the problem when a couple aims to optimize their consumption, portfolio and life-insurance purchasing strategies by maximizing the family objective until retirement.  

On the other hand, time non-separable preferences have gained in popularity on modeling consumption performance thanks to the capability to explain the observed consumption smoothness and equity premium puzzle. In the existing literature, there are two major types of time non-separable preferences involving the information of the past consumption path. The first type is the so-called habit formation preference, in which the utility is generated by the difference between the consumption rate and the weighted integral of the past consumption control; see  \cite{constantinides1990habit}. Along this direction, many notable studies in complete and incomplete market models can be found in Detemple and Zapatero \cite{DetempleZapatero1991Econometrica}, \cite{DetempleZapatero1992MF}, Detemple and Karatzas \cite{DetempleKaratzas2003JET}, Yu \cite{Yu2015AoAP}, \cite{Yu2017AoAP},
Yang and Yu \cite{YangYu2019arXiv}, Guan et al. \cite{GLY20}, He et al. \cite{HeLiang} among others. See also Angoshtari et al. \cite{AngoshtarBayraktar2022Sifin} where the habit formation is formulated as a control constraint. Some studies on optimal life insurance in the context of consumption habit formation can also be found, for example, in Ben-Arab et al. \cite{Ben} and Boyle et al. \cite{Boyle}.


The second type of preference chooses the past consumption maximum as the reference level. Guasoni et al. \cite{GuasoniHubermanR2020MFF} propose a shortfall aversion preference that measures the performance by the ratio of the consumption rate and the past spending maximum. Deng et al. \cite{DengLiPY2021FS} adopt the formulation from the habit formation preference where the utility is defined on the difference between the consumption rate and a proportion of the historical consumption maximum.
Later, Li et al. \cite{LiYuZ2021arXiv} extend the work of \cite{DengLiPY2021FS} to an S-shaped utility to account for agent's loss aversion towards the relative consumption with respect to the past consumption maximum.
Liang et al. \cite{LiangLuoY2022arXiv} generalize the preference in \cite{DengLiPY2021FS} such that the risk aversion differs when the consumption falls below the reference process and an additional drawdown constraint is enforced.
We also note some fruitful studies on the impact of the past consumption maximum when a ratcheting or a drawdown control constraint is considered under the standard time separable utility on consumption; see, for example, Angoshtari et al.  \cite{AngoshtarBayraktar2019Sifin}, Arun \cite{Arun2020arXiv}, Dybvig \cite{Dyb} , Jeon and Oh \cite{JeonOh2022JMAA}, Jeon and Park \cite{JeonPark2021MMOR}.

In this paper, we work with the second type of preference and choose the optimal relative consumption with reference to past spending maximum.
In particular, 
we adopt the shortfall aversion preference proposed in \cite{GuasoniHubermanR2020MFF} together with the dynamic life insurance control, {\red{}and enforce an additional drawdown constraint on consumption rate as a subsistence consumption requirement.}
{\red{}The objective function of the control problem also involves the expected bequest from life insurance, which renders the dimension reduction in \cite{GuasoniHubermanR2020MFF} not applicable in our problem. Instead, we encounter a two-dimensional HJB equation.} Similar to Deng et al. \cite{DengLiPY2021FS}, taking the wealth level and reference level as two state variables, we can derive the value function and optimal strategies in analytical form by solving the associated HJB inequality with some boundary conditions.
The HJB equation can be expressed in a piecewise form based on the decomposition of the state domain such that the feedback optimal consumption:
(1) equals the drawdown constraint rate;
(2) lies between the drawdown constraint and the past spending maximum;
(3) attains the past consumption peak.
By using the dual transform and some smooth-fit conditions, the HJB variational inequality is linearized to a parameterized ODE, which can be solved in closed-form.
The desired feedback form of optimal consumption, investment and insurance strategies can be obtained by the inverse transform. {\red Contrary to \cite{GuasoniHubermanR2020MFF}, our boundary curves for the wealth variable to distinguish different optimal feedback controls are all nonlinear functions due to the additional life insurance control.} Our analytical results allow us to numerically illustrate how the model parameters affect the optimal decision on consumption and life insurance. By comparing with some existing results without shortfall aversion, we can also illustrate how the reference of past spending maximum motivates the insurance purchase. Some interesting financial implications induced by the shortfall aversion preference and the drawdown constraint are discussed therein.


The remainder of this paper is organized as follows.
Section \ref{sec: formulate} introduces the market model with mortality risk and the stochastic control problem under the shortfall aversion preference.
Section \ref{sec: result} gives some heuristic arguments to solve the HJB variational inequality and present main results on the optimal feedback consumption, portfolio and life insurance controls. Section \ref{sec: numerical} presents several numerical examples to illustrate some sensitivity analysis results and their financial implications.
Some proofs are collected in Section \ref{sec: proof}.

\section{Model Setup and Problem Formulation}\label{sec: formulate}
\subsection{Market Model}
Let $(\Omega, \mathcal{F}, \mathbb{F}, \mathbb{P})$ be a filtered probability space and $\mathbb{F}=(\mathcal{F}_t)_{t\geq 0}$ satisfies the usual conditions.
The financial market model consists of one riskless asset and one risky asset.
The riskless asset price follows $dB_t=rB_tdt$, where $r>0$ is the interest rate.
The risky asset price is governed by the following stochastic differential equation (SDE)
\begin{equation}
dS_t=S_t\mu dt+S_t\sigma dW_t,\quad t\geq 0, \nonumber
\end{equation}
where $W$ is an $\mathbb{F}$-adapted Brownian motion, $\mu\in\mathbb{R}$ and $\sigma>0$ stand for the drift and volatility. It is assumed that $\mu>r$ and the sharp ratio is denoted by $\kappa:= \frac{\mu-r}{\sigma} > 0$.
{\red It is assumed that the individual's death time $\tau$ has an exponential distribution with the parameter $\lam>0$.}

Let $(\pi_t)_{t\geq 0}$ be the amount of wealth that the agent allocates in the risky asset, and let $(c_t)_{t\geq 0}$ represent the consumption rate.
Similar to \cite{Lee2021Math}, we assume that the life insurance contracts cover mortality risk and they are actuarially fair.
Denote by $p_t$ and $L_t$ the instantaneous life insurance premium rate paid by the individual and insurance benefit paid by the insurer, respectively.
We have that $p_t = \lam L_t$, and the bequest $b_t$ received by the individual's heir is given by $b_t = X_t+L_t = X_t + \frac{p_t}{\lam}$.
As a result, the wealth process satisfies
\begin{equation}\label{eq: system}
\begin{aligned}
dX_t&= \left(rX_t + \pi_t(\mu -r) - c_t - p_t\right)dt+\pi_t\sigma dW_t \\
&= \left((r+\lam)X_t + \pi_t(\mu -r) - c_t - \lam b_t\right)dt+\pi_t\sigma dW_t,
\quad t\geq 0,
\end{aligned}
\end{equation}
with the initial wealth $X_0=x\geq \nu H_0/(r+\lam)$.
A control variable $p_t$ is then transformed to the bequest $b_t$, which is assumed to be $\mathbb{F}$-adapted. The control triple $(c, \pi, b)$ is said to be \textit{admissible} if $c$ is $\mathbb{F}$-predictable and satisfies the drawdown constraint $c_t\geq \nu H_t$ where $\nu\in(0,1)$,
$\pi$ is $\mathbb{F}$-progressively measurable,
and $(c,\pi,b)$ satisfies the integrability condition $\int_0^{\infty} (c_t+\pi_t^2+b_t)dt<\infty$ a.s. and the no bankruptcy condition $X_t\geq 0$ a.s. for $t\geq 0$. Let $\mathcal{A}(x, h)$ denote the set of admissible controls $(c, \pi, b)$.

{\red In addition, similar to the proof of Corollary 1 of \cite{Arun2020arXiv}, to ensure that the consumption drawdown constraint $c_t\geq \nu H_t$ is sustainable for all $t\geq 0$, the necessary condition is $X_t\geq \frac{\nu H_t}{r+\lam}$ a.s. for all $t\geq 0$. Therefore, from this point onwards, we will only consider the feasible domain $(x,h)\in [0,+\infty)\times [0,+\infty)$ such that $x\geq \frac{\nu h}{r+\lam}$ for the admissible set  $\mathcal{A}(x, h)$.}

\begin{remark}\label{remark: insurance_negative}
The optimal premium $p_t$ is not required to be positive.
The wage earner is allowed to purchase a special term pension annuity, and she can receive the premium $p_t$ from the insurance company at time $t$.
However, the wage earner should pay $p_t$ to the company if she dies at time $t$.
This situation is related to the reverse mortgage.
Interested readers may refer to Pirvu and Zhang \cite{PirvuZhang2012IME} for more discussions.
\end{remark}

\subsection{Shortfall Aversion Preference and Control Problem}
It is assumed in the present paper that the agent is shortfall averse on relative consumption in the sense that utility loses of spending cuts from a reference.
The reference process is chosen as the running maximum consumption process $H_t:=\max{\{h,\ \sup_{s\leq t} c_s\}}$, and $H_0=h\geq 0$ is the initial reference level.
We adopt the shortfall aversion preference proposed in \cite{GuasoniHubermanR2020MFF} on consumption and also consider the expected utility on bequest at the death time. The objective function of the control problem is defined by
\begin{equation}\label{eq: value}
\begin{aligned}
&\mathbb{E}\left[\int_0^{\tau} e^{-\rho t}U(c_t, H_t)dt + e^{-\rho\tau} V(b_\tau) \right] \\
=& \mathbb{E}\left[\int_0^{\infty} e^{-(\rho+\lam) t}U(c_t, H_t)dt + \lam\int_0^\infty e^{-(\rho+\lam)t} V(b_t)dt \right],
\end{aligned}
\end{equation}
where $U(c,h)$ is the so-called shortfall aversion preference that satisfies
$$
U(c,h) = \begin{cases}
\frac1{\gamma_1} \big(\frac{c}{h^\alpha}\big)^{\gamma_1}, &\mbox{if } \nu h \leq c < h, \\
\frac1{\gamma_1} \big(c^{1-\alpha}\big)^{\gamma_1}, &\mbox{if } c \geq h,
\end{cases}
$$
with $0<\gamma_1<1$, and $V(b)$ is a standard CRRA utility 
%
%
that
$$
V(b) = K\frac{b^{\gamma_2}}{\gamma_2},~ 0<\gamma_2<1, ~K>0,
$$
and $K$ stands for the bequest motive level. According to Figure \ref{pic: utility}, the utility function $U(c,h)$ has a kink at $c=h$.
\begin{figure}[htbp]
\centering
\includegraphics[width=2.5in]{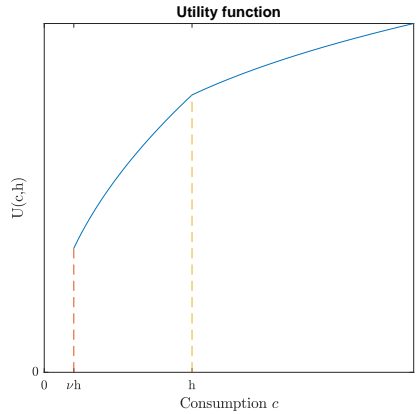}
\caption{Utility $U(c,h)$ for a consumption rate $c$, with reference point $h$}
\label{pic: utility}
\end{figure}

The agent aims to maximize the expected utility under shortfall aversion preference subjecting to a drawdown constraint on consumption control that
\begin{equation}\label{eq: problem}
\max\limits_{(c,\pi, b)\in\mathcal{A}(x,h)} \mathbb{E}\left[\int_0^{\infty} e^{-(\rho+\lam) t}U(c_t, H_t)dt + \lam\int_0^\infty e^{-(\rho+\lam)t} V(b_t)dt \right].
\end{equation}
For ease of presentation, it is assumed that the discount factor $\rho$ equals the risk-free rate $r$.
\section{Main Results}\label{sec: result}

\subsection{The HJB Equation}
For problem \eqref{eq: problem}, we can derive the auxiliary HJB variational inequality on the feasible domain $\{(x,h)\in [0,+\infty)\times [0,\infty): x\geq \frac{\nu h}{r+\lambda}\}$ using some heuristic arguments that

\begin{equation}\label{eq: HJB_eqn}
\begin{aligned}
\underset{c\in [\nu h,h], \pi\in\mathbb{R}, {\red{}b\geq 0}}{\sup}
\left[ -(r+\lam) u + u_x\left((r+\lam)x+\pi(\mu-r)-c-\lam b\right) + \frac{1}{2}\sigma^2\pi^2 u_{xx} + U(c,h)+\lam V(b) \right] &= 0,  \\
u_h(x,h)  & \leq  0,
\end{aligned}
\end{equation}
for $x\geq\frac{\nu h}{r+\lam}$ and $h \geq 0$.
The free boundary condition $u_h(x, h) = 0$ will be specified later.
Our goal is to find the optimal feedback control $c^*(x,h)$, $\pi^*(x,h)$, and $b^*(x,h)$.
If $u(x,\cdot)$ is $C^2$ in $x$, the first order condition gives the optimal portfolio and optimal bequest in feedback form that $\pi^{\ast}(x,h)=-\frac{\mu-r}{\sigma^2}\frac{u_x}{u_{xx}}$ and $b^*(x,h) = \big(\frac{u_x}{K}\big)^{\frac{1}{\gamma_2-1}}$, respectively.
The HJB variational inequality \eqref{eq: HJB_eqn} can be simplified to
\begin{equation} \label{eq: HJB_main}
\begin{aligned}
\sup_{c\in [\nu h,h]} \left[ U(c,h)- c u_x \right] - (r+\lam) u + (r+\lam)x u_x - \lam K^{-\frac{1}{\gamma_2-1}}\frac{1-\gamma_2}{\gamma_2}u_x^{\frac{\gamma_2}{\gamma_2-1}} -\frac{\kappa^2}{2}\frac{u_x^2}{u_{xx}} = 0, \\
 u_h\leq 0,\ \ \forall x\geq\frac{\nu h}{r+\lam}.
\end{aligned}
\end{equation}

\subsection{Some Heuristic Results}
We aim to solve the HJB variational inequality in the analytical form. In particular, we plan to characterize some thresholds (depending on $h$) for the wealth level $x$ such that the auxiliary value function, the optimal portfolio and consumption can be expressed analytically in each region.

Similar to Deng et al. \cite{DengLiPY2021FS} and Li et al. \cite{LiYuZ2021arXiv}, we can heuristically decompose the domain based on the first order condition with respect to $c$ and express the HJB equation \eqref{eq: HJB_main} piecewisely.
In particular, we have the following disjoint regions:

\emph{Region I}: on the set $\mathcal{R}_1 = \{(x,h) \in \mathbb{R}_+^2: x\geq\frac{\nu h}{r+\lam}, u_x > \nu^{\gamma_1-1}h^{(1-\alpha)\gamma_1-1}\},$
$U(c, h) - cu_x$ is decreasing in $c$ on $[\nu h,h]$, implying that $c^*=\nu h$.

\emph{Region II}: on the set
$\mathcal{R}_2 = \{(x,h)\in \mathbb{R}_+^2 : x\geq\frac{\nu h}{r+\lam}, h^{(1-\alpha)\gamma_1-1} \leq u_x\leq \nu^{\gamma_1-1}h^{(1-\alpha)\gamma_1-1}\},$
$U(c, h) - cu_x$ attains its maximum in $[\nu h ,h]$, implying that $c^*= h^{\frac{\alpha\gamma_1}{\gamma_1-1}}u_x^{\frac{1}{\gamma_1-1}}$.
%

\emph{Region III}: on the set $\mathcal{R}_3  = \{(x,h) \in \mathbb{R}_+^2 : x\geq\frac{\nu h}{r+\lam}, u_x < h^{(1-\alpha)\gamma_1-1}\}$,
$U(c, h) - cu_x$ is increasing in $c$ on $[\nu h,h]$, implying that $c^*=h$. To distinguish whether the optimal consumption $c_t^*$ updates the past maximum process $H_t^*$ in this region, we need to split \emph{Region III} in three subregions:

\emph{Region III-(i)}: on the set
$\mathcal{D}_1= \{(x,h) \in \mathbb{R}_+ \times \mathbb{R}_+ : x\geq \frac{\nu h}{r+\lam}, (1-\alpha)h^{(1-\alpha)\gamma_1-1} < u_x < h^{(1-\alpha)\gamma_1-1}\}$,
we have a contradiction that $\hat{c}(x)= \big(\frac{u_x(x,h)}{1-\alpha}\big)^{\frac{1}{(1-\alpha)\gamma_1-1}} < h$, and therefore $c_t^*$ is not a singular control.
We still need to follow the previous feedback form $c^*(x,h) = h$, in which $h$ is a previously attained maximum level.
The corresponding running maximum process remains flat at the instant time.
In this region, we only know that $u_h(x,h)\leq0$ as we have $dH_t = 0$.
%

\emph{Region III-(ii)}: on the set $D_2 := \{ (x,h)\in\mathbb{R}_+ \times \mathbb{R}_+ : x\geq\frac{\nu h}{r+\lam}, u_x = (1-\alpha)h^{(1-\alpha)\gamma_1-1}\}$, we get $\hat{c}(x) = \big(\frac{u_x(x,h)}{1-\alpha}\big)^{\frac{1}{(1-\alpha)\gamma_1-1}} = h$ and the feedback optimal consumption $c^*(x,h) = \big(\frac{u_x(x,h)}{1-\alpha}\big)^{\frac{1}{(1-\alpha)\gamma_1-1}}$.
This corresponds to the singular control $c_t^*$ that creates a new peak for the whole path and $H_t^* = c_t^* = \big(\frac{u_x(X_t^*,H_t^*)}{1-\alpha}\big)^{\frac{1}{(1-\alpha)\gamma_1-1}}$ is strictly increasing at the instant time so that $H_t^* > H_s^*$ for any $s<t$ and we must require the following free boundary condition that $u_h(x,h) = 0$.
In this region, it is noted that $c^*(x,h) = h$, therefore, the HJB equation follows the same PDE with in Region I but together with the new free boundary condition.

\emph{Region III-(iii)}: on the set $\mathcal{D}_3 := \{(x,h) \in\mathbb{R}_+ \times \mathbb{R}_+ : u_x(x,h) < (1-\alpha)h^{(1-\alpha)\gamma_1-1}\}$, we get $\hat{c}(x) = \big(\frac{u_x(x,h)}{1-\alpha}\big)^{\frac{1}{(1-\alpha)\gamma_1-1}} > h$.
This indicates that the initial reference level $h$ is below the feedback control $\hat{c}(x)$, and the optimal consumption is again a singular control $c^*(x) > h$, which creates a new consumption peak.
As the running maximum process $H_t^*$ is updated immediately by $c_t^*$, the feedback optimal consumption pulls the associated $H_{t-}^*$ upward from its original value to the new value in the direction of $h$ and $X_t^*$ remains the same, in which $u(x,h)$ is the solution of the HJB equation on the set $\mathcal{D}_2$.
This suggests that for any given initial value $(x,h)$ in the set $\mathcal{D}_3$, the feedback control $c^*(x,h)$ pushes the value function jumping immediately to the point $(x, \hat{h})$ on the boundary set $\mathcal{D}_2$ for the given level of $x$, where $\hat{h} = \big( \frac{u_x(x, \hat{h})}{1-\alpha} \big)^{\frac{1}{(1-\alpha)\gamma_1-1}}$.

Therefore, it is sufficient to consider the effective domain defined by
\begin{equation}\label{eq: effective domain}
\begin{aligned}
\mathcal{C} &:= \bigg\{(x,h) \in \mathbb{R}_+ \times \mathbb{R}_+ : x\geq \frac{\nu h}{r+\lam},
u_x(x,h) \geq (1-\alpha)h^{(1-\alpha)\gamma_1-1}\bigg\} \\
&= \mathcal{R}_1 \cup \mathcal{R}_2 \cup \mathcal{D}_1 \cup \mathcal{D}_2 \subset \mathbb{R}_+^2.
\end{aligned}\end{equation}
The only possibility for $(x,h) \in \mathcal{D}_3$ occurs at the initial time $t=0$. If $(X_0^*, H_0^*)$ starts from $\mathcal{C}$, then the controlled process $(X_t^*, H_t^*)$ always stay inside the region $\mathcal{C}$ and will either reflect at the boundary or move along the boundary $\mathcal{D}_2$ after visiting the boundary $\mathcal{D}_2$.
On the other hand, if the process $(X_0^*, H_0^*)$ starts from the value $(x,h)$ inside the region $\mathcal{D}_3$, the optimal control enforces an instant jump (and the only jump) of the process $H$ from $H_{0-} = h$ to $H_0 = \hat{h}$ on the boundary $\mathcal{D}_2$ and both processes $X_t$ and $H_t$ become continuous processes diffusing inside the effective domain $\mathcal{C}$ afterwards for $t<0$.

Therefore, the HJB variational inequality \eqref{eq: HJB_main} can be written as
\begin{equation}\label{eq: HJB_beta<1}
\begin{aligned}
- (r+\lam) u + (r+\lam)x u_x -\frac{\kappa^2}{2}\frac{u_x^2}{u_{xx}} = -\tV(u_x, h) , ~\mathrm{and}~ u_h\leq0, \\
u_h = 0, ~\mbox{if}~ u_x =(1-\alpha)h^{(1-\alpha)\gamma_1-1},
\end{aligned}
\end{equation}
where we define
\begin{equation}\label{eq: V_def}
\tV(q, h) := \begin{cases}
\lam K^{-\frac{1}{\gamma_2-1}}\frac{1-\gamma_2}{\gamma_2}q^{\frac{\gamma_2}{\gamma_2-1}} + \frac{\nu^{\gamma_1}}{\gamma}h^{(1-\alpha)\gamma_1} - \nu h q, & \mbox{if } q > \nu^{\gamma_1-1}h^{(1-\alpha)\gamma_1-1},  \\
\lam K^{-\frac{1}{\gamma_2-1}}\frac{1-\gamma_2}{\gamma_2}q^{\frac{\gamma_2}{\gamma_2-1}} + \frac{1-\gamma_1}{\gamma_1}h^{\frac{\alpha\gamma_1}{\gamma_1-1}}q^{\frac{\gamma_1}{\gamma_1-1}} , & \mbox{if } h^{(1-\alpha)\gamma_1-1} \leq q \leq \nu^{\gamma_1-1}h^{(1-\alpha)\gamma_1-1}, \\
\lam K^{-\frac{1}{\gamma_2-1}}\frac{1-\gamma_2}{\gamma_2}q^{\frac{\gamma_2}{\gamma_2-1}} + \frac{1}{\gamma_1}h^{(1-\alpha)\gamma_1} - hq, & \mbox{if } (1-\alpha)h^{(1-\alpha)\gamma_1-1} \leq q < h^{(1-\alpha)\gamma_1-1}.
\end{cases}
\end{equation}
To solve the equation, some boundary conditions are needed.
First, to guarantee the desired global regularity of the solution, we need to impose the smooth-fit condition along two free boundaries such that $u_x(x,h) = \nu^{\gamma_1-1}h^{(1-\alpha)\gamma_1-1}$ and $u_x(x,h) = h^{(1-\alpha)\gamma_1-1}$.
Next, note that if we start with initial wealth $x = \frac{\nu h}{r+\lam}$, to confront the risk of bankruptcy,
the optimal investment $\pi^*(x) = -\frac{\mu-r}{\sigma^2} \frac{u_x}{u_{xx}}$ should always be 0.
The wealth level will never change as there is no trading strategy, the consumption rate should also be $c_t = \nu h$, and the optimal bequest should also be 0 all the time.
Therefore, we can conclude that
\begin{equation}\label{eq: boundary_x->0}
\lim\limits_{x\rightarrow \frac{\nu h}{r+\lam}} \frac{u_x(x,h)}{u_{xx}(x,h)} = 0
~~\mathrm{and}~~
\lim\limits_{x\rightarrow \frac{\nu h}{r+\lam}} u(x, h) =  \int_0^{+\infty} e^{-(r+\lam)t} \frac1{\gamma_1} \bigg(\frac{\nu h}{h^\alpha}\bigg)^{\gamma_1}dt = \frac{\nu^{\gamma_1}}{(r+\lam)\gamma_1}h^{(1-\alpha)\gamma_1}.
\end{equation}
On the other hand, when the initial wealth tends to infinity, one can consume as much as possible that leads to the infinitely large consumption rate and bequest.
A small variation of initial wealth will only lead to a negligible change of the value function.
In addition, the optimal consumption rate should be proportional to the wealth level on region $\mathcal{D}_2$.
It follows that
\begin{equation}\label{eq: boundary_x->infty}
\lim\limits_{x\rightarrow+\infty} u_x(x,h) = 0,
~~\mathrm{and}~~
\lim\limits_{x\rightarrow+\infty,~ (x,h)\in\mathcal{D}_2} \frac{h}{x} = C_\infty,
\end{equation}
where $C_\infty >0$ is a constant.
See Corollary \ref{cor: asy_infty_wealth} for the verification of the last boundary condition .

To tackle the nonlinear HJB equation \eqref{eq: HJB_beta<1}, we employ the dual transform only with respect to the variable $x$ and treat the variable $h$ as a parameter; see similar dual transform arguments in Bo et al. \cite{BoLiY2021SiCon}, Deng et al. \cite{DengLiPY2021FS} and Li et al. \cite{LiYuZ2021arXiv}.
That is, we consider $v(y,h) := \sup_{x\geq\frac{\nu h}{r+\lam}}\{u(x,h) - xy\}$, $y\geq{\red(1-\alpha)h^{(1-\alpha)\gamma_1-1}}$.
For a given $(x,h)\in\mathcal{C}$, let us define the variable $y = u_x(x,h)$ and it holds that $u(x,h) = v(y,h) + xy$.
We can further deduce that
$$
x = -v_y(y,h), ~u(x,h) = v(y,h) - yv_y(y,h)~ \mathrm{and}~ u_{xx}(x,h) = -\frac{1}{v_{yy}(y,h)}.
$$
The nonlinear equation \eqref{eq: HJB_beta<1} can be reduced to
\begin{equation}\label{eq: LODE_beta<1}
\frac{\kappa^2}{2}y^2 v_{yy} - (r+\lam)v =  -\tV(y,h),
\end{equation}
where $\tV(\cdot,\cdot)$ is defined in \eqref{eq: V_def},
and the free boundary condition is transformed to the point $y = (1-\alpha)h^{(1-\alpha)\gamma_1-1}$.
As $h$ can be regarded as a parameter, we can study the above equation as the ODE problem of the variable $y$.
Based on the dual transform, the boundary conditions \eqref{eq: boundary_x->infty}  can be written as
\begin{equation}\label{eq: boundary_dual_y->0}
\lim\limits_{y\rightarrow0} v_y(y,h) = -\infty,  ~~ \mathrm{and} ~~
\lim\limits_{h\rightarrow\infty} \frac{h}{v_y(y,h)} = -C_\infty,
\end{equation}
on free boundary $y = (1-\alpha)h^{(1-\alpha)\gamma_1-1}$.
The boundary condition \eqref{eq: boundary_x->0} is equivalent to
\begin{equation}\label{eq: boundary_dual_y->infty}
yv_{yy}(y,h) \rightarrow 0
~~\mathrm{and}~~
v(y,h) - yv_{y}(y,h) \rightarrow \frac{\nu^{\gamma_1}}{(r+\lam)\gamma_1}h^{(1-\alpha)\gamma_1}
~~\mathrm{as}~~
v_y(y,h) \rightarrow -\frac{\nu h}{r+\lam}.
\end{equation}
It holds by the dual transform that $v_y(y,h) = -x$, and one can derive that
$u_h(x,h)
= v_h(y,h) + (v_y(y,h) + x)\frac{dy(h)}{dh}
= v_h(y,h)$. The free boundary condition \eqref{eq: HJB_beta<1} is written by
\begin{equation}\label{eq: free_boundary_dual}
 v_h(y,h)  = 0 ~~\mathrm{as}~~ y = (1-\alpha)h^{(1-\alpha)\gamma_1-1}.
\end{equation}

In particular, to facilitate some mathematical arguments, we need to impose the following technical assumption on model parameters. This assumption is needed in deriving the explicit form of coefficient functions $C_i(h)$, $i=1,...,6$, in Proposition \ref{prop: dual_solution_beta<1} below. It is also needed in the proof of Lemma \ref{lemma: beta<1} when we verify that the obtained solution $v(y,h)$ is  convex in the variable $y$ and in the proof of the verification theorem on optimality.\\

\noindent \textbf{Assumption (A1)} $\gamma_2 \leq (1-\alpha)\gamma_1 < -\frac{r_2}{r_1}\neq \gamma_1$, where $r_1>1$ and $r_2<0$ are two solutions to the equation $\eta^2 - \eta - \frac{2(r+\lam)}{\kappa^2} = 0$.

\begin{proposition}\label{prop: dual_solution_beta<1}
Under \textbf{Assumption (A1)}, boundary conditions \eqref{eq: boundary_dual_y->0}, \eqref{eq: boundary_dual_y->infty}, the free boundary condition \eqref{eq: free_boundary_dual},
and the smooth-fit conditions with respect to $y$ at free boundary points $y = \nu^{\gamma_1-1}h^{(1-\alpha)\gamma_1-1}$ and $y = h^{(1-\alpha)\gamma_1-1}$,
the ODE \eqref{eq: LODE_beta<1} in the domain $\{y\in\mathbb{R}: y\geq (1-\alpha)h^{(1-\alpha)\gamma_1-1}\}$ admits the unique solution given explicitly by
\begin{equation}\label{eq: dual_solu_beta<1}
v(y,h) = \begin{cases}
\begin{aligned}
&C_2(h)y^{r_2} + \frac{2\lam K^{1-\beta_2}}{\kappa^2\beta_2(\beta_2-r_1)(\beta_2-r_2)}y^{\beta_2}  \\
&+ \frac{\nu^{\gamma_1}}{(r+\lam)\gamma_1}h^{(1-\alpha)\gamma_1} - \frac{\nu h}{r+\lam} y,
\end{aligned}
& \mbox{if } y > \nu^{\gamma_1-1}h^{(1-\alpha)\gamma_1-1},   \\
\begin{aligned}
&C_3(h)y^{r_1} + C_4(h)y^{r_2} + \frac{2\lam K^{1-\beta_2}}{\kappa^2\beta_2(\beta_2-r_1)(\beta_2-r_2)}y^{\beta_2} \\
&+ \frac{2h^{\alpha\beta_1}}{\kappa^2 \beta_1(\beta_1-r_1)(\beta_1-r_2)}y^{\beta_1},
\end{aligned}
& \mbox{if } h^{(1-\alpha)\gamma_1-1} \leq y \leq \nu^{\gamma_1-1}h^{(1-\alpha)\gamma_1-1}, \\
\begin{aligned}
&C_5(h)y^{r_1} + C_6(h)y^{r_2} + \frac{2\lam K^{1-\beta_2}}{\kappa^2\beta_2(\beta_2-r_1)(\beta_2-r_2)}y^{\beta_2}  \\
&+ \frac{1}{(r+\lam)\gamma_1}h^{(1-\alpha)\gamma_1} - \frac{h}{r+\lam} y,
\end{aligned}
& \mbox{if } (1-\alpha)h^{(1-\alpha)\gamma_1-1} \leq y < h^{(1-\alpha)\gamma_1-1},  \\
\end{cases}
\end{equation}
where {\red$\beta_1 = \frac{\gamma_1}{\gamma_1-1}$, $\beta_2 = \frac{\gamma_2}{\gamma_2-1}$}, and functions $C_2(h), C_3(h), \cdots, C_6(h)$ are given by
\begin{equation}\label{eq: C_def}
\begin{aligned}
C_2(h)  &= C_4(h) + \frac{1-\beta_1}{(r+\lam)(r_1-r_2)(\beta_1-r_2)}\nu^{r_1\gamma_1+r_2}h^{r_1(1-\alpha)\gamma_1+r_2},\\
C_3(h) &= \frac{1-\beta_1}{(r+\lam)(r_1-r_2)(\beta_1-r_1)}\nu^{r_2\gamma_1+r_1}h^{r_2(1-\alpha)\gamma_1+r_1},\\
C_4(h)  &=  C_6(h) + \frac{\beta_1-1}{(r+\lam)(r_1-r_2)(\beta_1-r_2)}h^{r_1(1-\alpha)\gamma_1+r_2},\\
C_5(h) &= C_3(h) - \frac{1-\beta_1}{(r+\lam)(r_1-r_2)(\beta_1-r_1)}h^{r_2(1-\alpha)\gamma_1+r_1},\\
C_6(h) &= \frac{(1-\alpha)^{r_1-r_2}(1-\beta_1)(r_2(1-\alpha)\gamma_1+r_1)}{(r+\lam)(r_1-r_2)(\beta_1-r_1)(r_1(1-\alpha)\gamma_1+r_2)}(1-\nu^{r_2\gamma_1+r_1})h^{r_1(1-\alpha)\gamma_1+r_2}
\end{aligned}
\end{equation}
where $r_1>1$ and $r_2<0$ are two roots to the quadratic equation $\eta^2 - \eta - \frac{2(r+\lam)}{\kappa^2} = 0$.
\end{proposition}
{\red
\begin{proof}
It is easy to show that the general solution of the linear ODE \eqref{eq: LODE_beta<1} admits the piecewise form in each region that
\begin{equation}\label{eq: gen_solu_dual_beta<1}
v(y,h) = \begin{cases}
C_1(h)y^{r_1} + C_2(h)y^{r_2} + \frac{(\nu h)^{\gamma_1}}{(r+\lam)\gamma_1 h^{\alpha\gamma_1}} - \frac{\nu h}{r+\lam}y, & \mbox{if }  > \nu^{\gamma_1-1}h^{(1-\alpha)\gamma_1-1} \\
C_3(h)y^{r_1} + C_4(h)y^{r_2} + \frac{2h^{\alpha\beta_1}}{\kappa^2\beta_1(\beta_1-r_1)(\beta_1-r_2)}y^{\beta_1}, & \mbox{if } h^{(1-\alpha)\gamma_1-1} \leq y \leq \nu^{\gamma_1-1}h^{(1-\alpha)\gamma_1-1}, \\
C_5(h)y^{r_1} + C_6(h)y^{r_2} + \frac{1}{(r+\lam)\gamma_1}h^{(1-\alpha)\gamma} - \frac{h}{r+\lam}y, & \mbox{if } (1-\alpha)h^{(1-\alpha)\gamma_1-1} \leq y < h^{(1-\alpha)\gamma_1-1},  \\
\end{cases}
\end{equation}
where $C_1(\cdot), \cdots, C_6(\cdot)$ are functions of $h$ to be determined.

The free boundary condition $v_y(y,h)\rightarrow-\frac{\nu h}{r+\lam}$ in \eqref{eq: boundary_dual_y->infty} implies that $y\rightarrow+\infty$.
Together with free boundary conditions in \eqref{eq: boundary_dual_y->infty} and the formula of $v(y,h)$ in the region $y> \nu^{\gamma_1-1}h^{(1-\alpha)\gamma_1-1}$, we deduce $C_1(h)\equiv 0$. To determine the left parameters, we consider the smooth-fit conditions with respect to the variable $y$ at two free boundary points $y = y_1(h) = \nu^{\gamma_1-1}h^{(1-\alpha)\gamma_1-1}$ and $y = y_2(h) = h^{(1-\alpha)\gamma-1}$, that is,
\begin{equation}\label{eq: dual_smooth_beta<1}
\begin{aligned}
&-C_3(h)y_1(h)^{r_1} + (C_2(h)-C_4(h))y_1(h)^{r_2} \\
&= \frac{2h^{\alpha\beta_1}}{\kappa^2\beta_1(\beta_1-r_1)(\beta_1-r_2)}y_1(h)^{\beta_1} + \frac{\nu h}{r+\lam}y_1(h) - \frac{(\nu h)^\gamma}{(r+\lam)\gamma h^{\alpha\gamma}}, \\
&-r_1C_3(h)y_1(h)^{r_1-1} + r_2(C_2(h)-C_4(h))y_1(h)^{r_2-1} \\
&= \frac{2h^{\alpha\beta_1}}{\kappa^2(\beta_1-r_1)(\beta_1-r_2)}y_1(h)^{\beta_1-1} + \frac{\nu h}{r+\lam}, \\
&(C_3(h)-C_5(h))y_2(h)^{r_1} + (C_4(h)-C_6(h))y_2(h)^{r_2}  \\
&= -\frac{2h^{\alpha\beta_1}}{\kappa^2\beta_1(\beta_1-r_1)(\beta_1-r_2)}y_2(h)^{\beta_1} + \frac{1}{(r+\lam)\gamma}h^{(1-\alpha)\gamma} - \frac{h}{r+\lam}y_2(h) , \\
& r_1(C_3(h)-C_5(h))y_2(h)^{r_1-1} + r_2(C_4(h)-C_6(h))y_2(h)^{r_2-1} \\
&= -\frac{2h^{\alpha\beta_1}}{\kappa^2(\beta_1-r_1)(\beta_1-r_2)}y_2(h)^{\beta_1-1} - \frac{h}{r+\lam}.
\end{aligned}
\end{equation}
Then the equations \eqref{eq: dual_smooth_beta<1} are linear equations for $C_3(h)$, $C_2(h) - C_4(h)$, and $C_3(h)-C_5(h)$ and $C_4(h) - C_6(h)$.
By solving the above two systems, we can obtain
\begin{equation}\label{eq: solu_dual_smooth_beta<1}
\begin{aligned}
C_3(h) &= \frac{1-\beta_1}{(r+\lam)(r_1-r_2)(\beta_1-r_1)}(\nu h)^{r_2\gamma_1+r_1}h^{-r_2\alpha\gamma_1},\\
C_2(h) - C_4(h) &= \frac{1-\beta_1}{(r+\lam)(r_1-r_2)(\beta_1-r_2)}(\nu h)^{r_1\gamma_1+r_2}h^{-r_1\alpha\gamma_1},\\
C_3(h) - C_5(h) &= \frac{1-\beta_1}{(r+\lam)(r_1-r_2)(\beta_1-r_1)}h^{r_2(1-\alpha)\gamma_1+r_1},\\
C_4(h) - C_6(h) &= \frac{\beta_1-1}{(r+\lam)(r_1-r_2)(\beta_1-r_2)}h^{r_1(1-\alpha)\gamma_1+r_2},
\end{aligned}
\end{equation}
therefore, $C_2(h)$ to $C_5(h)$ can be written by \eqref{eq: dual_solu_beta<1}.

To obtain $C_2(h),C_4(h)$ and $C_6(h)$, we aim to find $C_6(h)$ first, then $C_4(h)$ and $C_2(h)$ can be determined.
Indeed, as $h\rightarrow+\infty$, we get $y\rightarrow0$ in the region $(1-\alpha)h^{(1-\alpha)\gamma_1-1} \leq y < h^{(1-\alpha)\gamma_1-1}$, and the boundary condition \eqref{eq: boundary_dual_y->0} leads to
$$
\lim\limits_{h\rightarrow+\infty }
\frac{h}{v_y((1-\alpha)h^{(1-\alpha)\gamma_1-1},h)} = C,
$$
where $C$ is a negative constant.
Along the free boundary, we have
$$
v_y((1-\alpha)h^{(1-\alpha)\gamma_1-1},h) = r_1C_5(h)\big((1-\alpha)h^{(1-\alpha)\gamma_1-1}\big)^{r_1-1} + r_2C_6(h)\big((1-\alpha)h^{(1-\alpha)\gamma_1-1}\big)^{r_2-1} + \frac{h}{r+\lam}.
$$
It follows from $\lim\limits_{h\rightarrow+\infty }
\frac{h}{v_y((1-\alpha)h^{(1-\alpha)\gamma-1},h)} < 0$ that $v_y((1-\alpha)h^{(1-\alpha)\gamma-1}, h) = O(h)$ as $h\rightarrow+\infty$.
Therefore, we can deduce that
$$
C_6(h) = O(C_5(h)h^{(r_1-r_2)((1-\alpha)\gamma_1-1)}) + O(h^{r_1(1-\alpha)\gamma_1+r_2}).
$$
From the asymptotic property of $C_5(h)$ in Lemma \ref{lemma: C_order}, it follows that
$$\begin{aligned}
C_6(h) &= O(C_5(h)h^{(r_1-r_2)((1-\alpha)\gamma_1-1)}) + O(h^{r_1(1-\alpha)\gamma_1+r_2})= O(h^{r_1(1-\alpha)\gamma_1+r_2}),
\end{aligned}$$
as $h\rightarrow+\infty$.
By \textbf{Assumption (A1)}, we have $\lim\limits_{h\rightarrow+\infty} C_6(h) = 0$, and thus we have $C_6(h) = -\int_h^\infty C_6'(s)ds$.

In addition, to obtain $C_6'(h)$, we apply the free boundary condition \eqref{eq: free_boundary_dual} at point $y = (1-\alpha)h^{(1-\alpha)\gamma_1-1}$ that
$$
C'_5(h)\big((1-\alpha)h^{(1-\alpha)\gamma_1-1}\big)^{r_1} + C'_6(h)\big((1-\alpha)h^{(1-\alpha)\gamma_1-1}\big)^{r_2} + \frac{1-\alpha}{r+\lam}h^{(1-\alpha)\gamma_1-1} - \frac{1-\alpha}{r+\lam}h^{(1-\alpha)\gamma_1-1} = 0,
$$
which yields that
$$
\begin{aligned}
C'_6(h)
&= -(1-\alpha)^{r_1-r_2} C'_5(h) h^{(r_1-r_2)((1-\alpha)\gamma_1-1)} \\
&= \frac{(1-\alpha)^{r_1-r_2}(1-\beta_1)(r_2(1-\alpha)\gamma_1+r_1)}{(r+\lam)(r_1-r_2)(\beta_1-r_1)}(1-\nu^{r_2\gamma_1+r_1})h^{r_1((1-\alpha)\gamma_1-1)}
.
\end{aligned}
$$
As a result, we conclude that
$$
C_6(h) = -\int_h^\infty C'_6(s)ds = \frac{(1-\alpha)^{r_1-r_2}(1-\beta_1)(r_2(1-\alpha)\gamma_1+r_1)}{(r+\lam)(r_1-r_2)(\beta_1-r_1)(r_1(1-\alpha)\gamma_1+r_2)}(1-\nu^{r_2\gamma_1+r_1})h^{r_1(1-\alpha)\gamma_1+r_2}.
$$
\end{proof}
}

\begin{theorem}[Verification Theorem]\label{thm: beta<1}
Let $(x,h) \in \mathcal{C}$, $h\in\mathbb{R}$ and $0<\lam <1$, where $x\geq0$ stands for the initial wealth, $h\geq0$ is the initial reference level, and $\mathcal{C}$ stands for the effective domain \eqref{eq: effective domain}.
For $(y,h)\in\{(y,h)\in\mathbb{R}_+^2: y \geq  {\red (1-\alpha)h^{(1-\alpha)\gamma_1-1}}\}$, let us define the feedback functions that
\begin{equation}\label{eq: c_y_beta<1}
c^\dag(y,h) = \begin{cases}
\nu h , & \mbox{if } y > \nu^{\gamma_1-1}h^{(1-\alpha)\gamma_1-1},  \\
h^{\frac{\alpha\gamma_1}{\gamma_1-1}}u_x^{\frac{1}{\gamma_1-1}}, & \mbox{if } h^{(1-\alpha)\gamma_1-1} \leq y \leq \nu^{\gamma_1-1}h^{(1-\alpha)\gamma_1-1}, \\
h, & \mbox{if } (1-\alpha)h^{(1-\alpha)\gamma_1-1} < y < h^{(1-\alpha)\gamma_1-1}, \\
\big(\frac{y}{1-\alpha}\big)^{\frac{1}{(1-\alpha)\gamma_1-1}}, & \mbox{if } y = (1-\alpha)h^{(1-\alpha)\gamma_1-1},
\end{cases}
\end{equation}
\begin{equation}\label{eq: pi_y_beta<1}
\begin{aligned}
&\pi^\dag(y,h) = \frac{\mu-r}{\sigma^2}y v_{yy}(y,h) \\
= & \frac{\mu - r}{\sigma^2}
\begin{cases}
\frac{2(r+\lam)}{\kappa^2}C_2(h)y^{r_2-1} + \frac{2\lam K^{1-\beta_2}(\beta_2-1)}{\kappa^2(\beta_2-r_1)(\beta_2-r_2)}y^{\beta_2-1}, & \mbox{if } y > \nu^{\gamma_1-1}h^{(1-\alpha)\gamma_1-1}, \\
\begin{aligned}
 &\frac{2(r+\lam)}{\kappa^2}C_3(h)y^{r_1-1} +  \frac{2(r+\lam)}{\kappa^2}C_4(h)y^{r_2-1} \\
&+ \frac{2\lam K^{1-\beta_2}(\beta_2-1)}{\kappa^2(\beta_2-r_1)(\beta_2-r_2)}y^{\beta_2-1}\\
&+ \frac{2(\beta_1-1)h^{\alpha\gamma_1}}{\kappa^2(\beta_1-r_1)(\beta_1-r_2)}y^{\beta_1-1},
\end{aligned}
& \mbox{if } h^{(1-\alpha)\gamma_1-1} \leq y \leq \nu^{\gamma_1-1}h^{(1-\alpha)\gamma_1-1},\\
\begin{aligned}
&\frac{2(r+\lam)}{\kappa^2}C_5(h)y^{r_1-1} +  \frac{2(r+\lam)}{\kappa^2}C_6(h)y^{r_2-1} \\
&+ \frac{2\lam K^{1-\beta_2}(\beta_2-1)}{\kappa^2(\beta_2-r_1)(\beta_2-r_2)}y^{\beta_2-1} ,
\end{aligned}
& \mbox{if } (1-\alpha)h^{(1-\alpha)\gamma_1-1} \leq y < h^{(1-\alpha)\gamma_1-1},
\end{cases}
\end{aligned}
\end{equation}
and
\begin{equation}
b^\dag(y,h) = \bigg(\frac{y}{K}\bigg)^{\frac{1}{\gamma_2-1}}.
\end{equation}
We consider the process $Y_t(y) := y e^{(r+\lam)t}M_t$, where $M_t := e^{-(r+\lam+\frac{\kappa^2}{2})t - \kappa W_t}$ is the discounted rate state price density process, and $y^* = y^*(x,h)$ is the unique solution to the budget constraint $\mathbb{E}[ \int_0^\infty (c^\dag(Y_t(y), H_t^\dag(y)) + \lam b^\dag(Y_t(y), H_t^\dag(y)))$$M_tdt ] = x$, where
$$
H_t^\dag(y) := h\vee \sup\limits_{s\leq t} c^\dag(Y_s(y), H_s^\dag(y))
=h\vee \left(\inf\limits_{s\leq t}Y_s(y)/(1-\alpha)\right)^{\frac{1}{(1-\alpha)\gamma_1-1}},
$$
is the optimal reference process corresponding to any fixed $y>0$.
The value function $u(x,h)$ can be attained by employing the optimal consumption and portfolio strategies in the feedback form that $c_t^* = c^\dag(Y_t^*, H_t^*)$ and $\pi_t^* = \pi^\dag(Y_t^*, H_t^*)$ for all $t\geq0$, where $Y_t^*:=Y_t(y^*)$ and $H_t^*=H_t^\dag(y^*)$.

The process $H_t^*$ is strictly increasing if and only if $Y_t^*=(1-\alpha){H_t^*}^{(1-\alpha)\gamma_1-1}$.
If we have $y^*(x,h) < (1-\alpha)h^{(1-\alpha)\gamma_1-1}$ at the initial time, the optimal consumption creates a new peak and brings $H_{0-}^* = h$ jumping immediately to a higher level $H_0^* = \big(\frac{y^*(x,h)}{1-\alpha}\big)^{\frac{1}{(1-\alpha)\gamma_1-1}}$ such that $t=0$ becomes the only jump time of $H_t^*$.
\end{theorem}

\begin{proof}
Similar to Deng et al. \cite{DengLiPY2021FS}, we need to show that the solution of the HJB equation \eqref{eq: HJB_eqn} coincides with the value function, i.e. there exists $(\pi^*, c^*, b^*)\in\mathcal{A}(x)$ such that
$u(x,h) = \mathbb{E}\bigg[ \int_0^\infty e^{-rt}u(c_t^*, H_t^*)dt \bigg]$. For any admissible strategy $(\pi, c) \in \mathcal{A}(x)$, similar to the proof of Lemma 1 in \cite{Arun2020arXiv}, we have
\begin{equation}\label{eq: ECM<=x}
\mathbb{E}\bigg[ \int_0^\infty (c_t+\lam b_t)M_tdt \bigg] \leq x.
\end{equation}

Let $h$ be the fixed parameter, the dual transform of $U(c,h){\red+\lam V(b)}$ with respect to $c$ {\red and $b$} in the constrained domain that {\red$\widetilde{V}(q,h) := \sup_{c\in[\nu h, h]} \big[U(c,h)-cq\big] + \lam\sup_{b\geq0}\big[V(b) - bq\big]$} defined in \eqref{eq: V_def}.
Moreover, $\widetilde{V}$ can be attained by the construction of the feedback optimal control $c^\dag(y,h)$ in \eqref{eq: c_y_beta<1}.

In what follows, we distinguish the two reference processes, namely $H_t := h\vee \sup_{s\leq t}c_s$ and $H_t^\dag(y):= h\vee \sup_{s\leq t} c^\dag(Y_s(y), H_s^\dag(y))$ that correspond to the reference process under an arbitrary consumption process $c_t$ and under the optimal consumption process $c^\dag$ with an arbitrary $y>0$.
Note that the global optimal reference process will be defined later by $H_t^* := H_t^\dag(y^*)$ with $y^*>0$ to be determined.
Let us now further introduce
\begin{equation}\label{eq: H_hat}
\hat{H}_t(y) := h\vee \bigg( (1-\alpha)^{-\frac{1}{(1-\alpha)\gamma_1-1}} \big(\inf_{s\leq t}Y_s(y) \big)^{\frac{1}{(1-\alpha)\gamma_1-1}} \bigg),
\end{equation}
where $Y_t(y) = ye^{rt}M_t$ is the discounted martingale measure density process.

For any admissible controls $(\pi, c) \in \mathcal{A}(x)$, recall the reference process $H_t = h\vee \sup_{s\leq t}c_s$, and for all $y>0$, we see that
\begin{equation}\label{eq: verification_main}
\begin{aligned}
&~ \mathbb{E}\left[\int_0^{\infty} e^{-(r+\lam) t}U(c_t, H_t)dt + \lam\int_0^\infty e^{-(r+\lam)t} V(b_t)dt \right]\\
&= {\red\mathbb{E}\bigg[ \int_0^\infty e^{-(r+\lam)t}(U(c_t, H_t) - Y_t(y)c_t)dt\bigg]+ \lambda\mathbb{E}\bigg[ \int_0^\infty e^{-(r+\lam)t}(V(b_t)-Y_t(y)b_t )dt \bigg]}\\
&\quad+ y\mathbb{E}\bigg[ \int_0^\infty (c_t+\lam b_t)M_tdt \bigg] \\
&\leq \mathbb{E}\bigg[ \int_0^\infty e^{-(r+\lam)t} \tV(Y_t(y), H_t^\dag(y))dt \bigg]
 + yx \\
&= \mathbb{E}\bigg[ \int_0^\infty e^{-(r+\lam)t} \tV(Y_t(y), \hat{H}_t(y))dt \bigg]
+ yx \\
&= v(y,h) + yx,
\end{aligned}
\end{equation}
the third equation holds because of Lemma \ref{lemma: 2}, and the last equation is verified by Lemma \ref{lemma: 1}.
In addition, Lemma \ref{lemma: 3} guarantees the inequality, and shows it becomes an equality with the choices of $c_t^* = c^\dag(Y_t(y^*), H_t^\dag(y^*))$ and $b_t^*= b^\dag(Y_t(y^*), H_t^\dag(y^*))$
, in which $y^*$ is the solution to the equation $\mathbb{E}\big[ \int_0^\infty (c^\dag(Y_t(y^*), H_t^\dag(y^*))+\lam b^\dag(Y_t(y^*), H_t^\dag(y^*)))M_tdt \big] = x$ for a given $x\geq \frac{\nu h}{r+\lam}$. In conclusion, we have
\begin{align*}
&\sup\limits_{(\pi, c)\in\mathcal{A}(x)}
\mathbb{E}\left[\int_0^{\infty} e^{-(r+\lam) t}U(c_t, H_t)dt + \lam\int_0^\infty e^{-(r+\lam)t} V(b_t)dt \right]\\
=& \inf_{y>0}(v(y,h) + yx) = u(x,h).
\end{align*}
\end{proof}


Using the dual relationship between $u$ and $v$, we have the optimal $x = g(\cdot, h) := -v_y(\cdot,h)$.
Define $f(\cdot, h)$ as the inverse of $g(\cdot, h)$, then $u(x,h) = v(f(x,h), h) + xf(x,h)$.
Note that $v$ has different expressions in the regions $c=0$, $0<c<h$ and $c=h$, the function $f$ should also have the piecewise from across these regions.
By the definition of $g$, the invertibility of the map $x\mapsto g(x,h)$ is guaranteed by the following lemma.

\begin{lemma}\label{lemma: beta<1}
Under \textbf{Assumption (A1)}, the value function $v(y,h)$ in \eqref{eq: dual_solu_beta<1} is convex in all regions so that the inverse Legendre transform $u(x,h) = \inf_{y\geq (1-\alpha)h^{(1-\alpha)\gamma-1}} [v(y,h) + xy]$ is well defined.
Moreover, it implies that the feedback optimal portfolio $\pi^*(y,h) > 0$ at all time.
\end{lemma}
\begin{proof}
See Section \ref{sec: proof_convex_dual}.
\end{proof}

\subsection{Optimal Feedback Controls}
The main result in this subsection is based on \textbf{Assumption (A1)}. Thanks to Lemma \ref{lemma: beta<1}, we can apply the inverse Legendre transform to the solution $v(y,h)$ in \eqref{eq: dual_solu_beta<1}. Similar to Section 3.1 in \cite{DengLiPY2021FS}, we can characterize the following four boundary curves $x_\mathrm{bound}(h)$, $x_\mathrm{low}(h)$, $x_\mathrm{aggr}(h)$, and $x_\mathrm{lavs}(h)$ that
\begin{equation}\label{eq: boundcurves_x}
\begin{aligned}
x_{\mathrm{bound}}(h) &:= \frac{\nu h}{r+\lam}, \\
x_{\mathrm{low}}(h) &:= -C_2(h)r_2 \nu^{-r_1(\gamma_1-1)}h^{-r_1((1-\alpha)\gamma_1-1)} - \frac{2\lam K^{1-\beta_2}\nu^{(\beta_2-1)(\gamma_1-1)}}{\kappa^2(\beta_2-r_1)(\beta_2-r_2)}h^{(\beta_2-1)((1-\alpha)\gamma_1-1)} + \frac{\nu h}{r+\lam}, \\
x_{\mathrm{aggr}}(h) &:= -C_3(h)r_1h^{-r_2((1-\alpha)\gamma_1-1)} - C_4(h)r_2h^{-r_1((1-\alpha)\gamma_1-1)}
\\ &~~~~-\frac{2\lam K^{1-\beta_2}}{\kappa^2(\beta_2-r_1)(\beta_2-r_2)}h^{(\beta_2-1)((1-\alpha)\gamma_1-1)}  - \frac{2}{\kappa^2 (\beta_1-r_1)(\beta_1-r_2)}h, \\
x_{\mathrm{lavs}}(h) &:= -C_5(h)r_1(1-\alpha)^{r_1-1}h^{-r_2((1-\alpha)\gamma_1-1)} - C_6(h)r_2(1-\alpha)^{r_2-1}h^{-r_1((1-\alpha)\gamma_1-1)}
\\ &~~~~-\frac{2\lam(1-\alpha)^{\beta_2-1}K^{1-\beta_2}}{\kappa^2(\beta_2-r_1)(\beta_2-r_2)}h^{(\beta_2-1)((1-\alpha)\gamma_1-1)} + \frac{h}{r+\lam},
\end{aligned}
\end{equation}
{\red and it holds that the feedback function of the optimal consumption satisfies: (i) $c^*(x,h)=\nu h$ when $x_\mathrm{bound}(h) \leq x < x_\mathrm{low}(h)$; (ii) $\nu h<c^*(x,h)<h$ when $x_\mathrm{low}(h) \leq x \leq x_\mathrm{aggr}(h)$; (iii) $c^*(x,h)=h$ when $x_\mathrm{aggr}(h) < x \leq x_\mathrm{lavs}(h)$. In particular, the condition $u_x(x,h)\geq (1-\alpha)h^{(1-\alpha)\gamma_1-1}$ in the effective domain $\mathcal{C}$ in \eqref{eq: effective domain} now can be explicitly expressed as $x\leq x_\mathrm{lavs}(h)$.}

%
We also define functions $f_1(x,h)$, $f_2(x,h)$ and $f_3(x,h)$ to be the respective solutions to three equations that
\begin{equation}\label{eq: def_f(x,h)}
\begin{aligned}
x &=  -C_2(h)r_2(f_1(x,h))^{r_2-1} -\frac{2\lam K^{1-\beta_2}f_1(x,h)^{\beta_2-1}}{\kappa^2(\beta_2-r_1)(\beta_2-r_2)} + \frac{\nu h}{r+\lam}, &\mbox{if } x_\mathrm{bound}(h) \leq x < x_\mathrm{low}(h), \\
x &= -C_3(h)r_1(f_2(x,h))^{r_1-1} - C_4(h)r_2(f_2(x,h))^{r_2-1} \\
&~~~-\frac{2\lam K^{1-\beta_2}f_2(x,h)^{\beta_2-1}}{\kappa^2(\beta_2-r_1)(\beta_2-r_2)} - \frac{2h^{\alpha\beta_1}f_2(x,h)^{\beta_1-1}}{\kappa^2(\beta_1-r_1)(\beta_1-r_2)},
  &\mbox{if } x_\mathrm{low}(h) \leq x \leq x_\mathrm{aggr}(h) , \\
x &= -C_5(h)r_1(f_3(x,h))^{r_1-1} - C_6(h)r_2(f_3(x,h))^{r_2-1} \\
&~~~-\frac{2\lam K^{1-\beta_2}f_3(x,h)^{\beta_2-1}}{\kappa^2(\beta_2-r_1)(\beta_2-r_2)} + \frac{h}{r+\lam},
 &\mbox{if } x_\mathrm{aggr}(h) < x \leq x_\mathrm{lavs}(h).
\end{aligned}
\end{equation}
%

The following proposition shows the semi-analytical form for the value function, optimal consumption, and optimal portfolio.
\begin{theorem}\label{cor: beta<1}
For $(x,h) \in \{(x,h)\in\mathbb{R}_+^2: x\geq x_\mathrm{bound}(h)\}$, $0<\nu<1$, $\gamma_1, \gamma_2>0$, the value function $u(x,h)$ in \eqref{eq: value} can be expressed in a piecewise form that
\begin{equation}\label{eq: u_beta<1}
\small
u(x,h) = \begin{cases}
\begin{aligned}
&C_2(h)f_1(x,h)^{r_2} + \frac{2\lam K^{1-\beta_2}}{\kappa^2\beta_2(\beta_2-r_1)(\beta_2-r_2)}f_1(x,h)^{\beta_2}  \\
&+ \frac{\nu^\gamma}{(r+\lam)\gamma_1}h^{(1-\alpha)\gamma} - \frac{\nu h}{r+\lam} f_1(x,h),
\end{aligned}
& \mbox{if } x_\mathrm{bound}(h) \leq x < x_\mathrm{low}(h),   \\
\begin{aligned}
&C_3(h)f_2(x,h)^{r_1} + C_4(h)f_2(x,h)^{r_2} + \frac{2\lam K^{1-\beta_2}}{\kappa^2\beta_2(\beta_2-r_1)(\beta_2-r_2)}f_2(x,h)^{\beta_2} \\
&+ \frac{2h^{\alpha\gamma_1}}{\kappa^2 \beta_1(\beta_1-r_1)(\beta_1-r_2)}f_2(x,h)^{\beta_1},
\end{aligned}
& \mbox{if } x_\mathrm{low}(h) \leq x \leq x_\mathrm{aggr}(h), \\
\begin{aligned}
&C_5(h)f_3(x,h)^{r_1} + C_6(h)f_3(x,h)^{r_2} + \frac{2\lam K^{1-\beta_2}}{\kappa^2\beta_2(\beta_2-r_1)(\beta_2-r_2)}f_3(x,h)^{\beta_2}  \\
&+ \frac{1}{(r+\lam)\gamma_1}h^{(1-\alpha)\gamma_1} - \frac{h}{r+\lam} f_3(x,h),
\end{aligned}
& \mbox{if } x_\mathrm{aggr}(h) < x \leq x_\mathrm{lavs}(h),  \\
\end{cases}
\end{equation}
where the boundaries $x_\mathrm{bound}(h)$, $x_\mathrm{low}(h)$, $x_\mathrm{aggr}(h)$, and $x_\mathrm{lavs}(h)$ are given in \eqref{eq: boundcurves_x}.
Moreover, the feedback optimal consumption and portfolio can also be given in terms of primal variables $(x,h)$ accordingly:
\begin{equation}\label{eq: c_beta<1}
c^*(x,h) = \begin{cases}
\nu h , & \mbox{if } x_\mathrm{bound}(h) \leq x < x_\mathrm{low}(h),  \\
h^{\frac{\alpha\gamma_1}{\gamma_1-1}}f_2(x,h)^{\frac{1}{\gamma_1-1}}, & \mbox{if } x_\mathrm{low}(h) \leq x \leq x_\mathrm{aggr}(h), \\
h, & \mbox{if } x_\mathrm{aggr}(h) < x < x_\mathrm{lavs}(h), \\
\big(\frac{f_3(x,\tilde{h}(x))}{1-\alpha}\big)^{\frac{1}{(1-\alpha)\gamma_1-1}}, & \mbox{if } x = x_\mathrm{lavs}(h),
\end{cases}
\end{equation}
where  $\tilde{h}(x):= x^{-1}_\mathrm{lavs}(x)$, the optimal portfolio
\begin{equation}\label{eq: pi_beta<1}
\begin{aligned}
&\pi^*(x,h) \\
=&\frac{\mu-r}{\sigma^2}
\begin{cases}
 \frac{2(r+\lam)}{\kappa^2}C_2(h)f_1(x,h)^{r_2-1} + \frac{2\lam K^{1-\beta_2}(\beta_2-1)}{\kappa^2(\beta_2-r_1)(\beta_2-r_2)}f_1(x,h)^{\beta_2-1}, & \mbox{if } x_\mathrm{bound}(h) \leq x < x_\mathrm{low}(h), \\
\begin{aligned}
 &\frac{2(r+\lam)}{\kappa^2}C_3(h)f_2(x,h)^{r_1-1} +  \frac{2(r+\lam)}{\kappa^2}C_4(h)f_2(x,h)^{r_2-1} \\
&+ \frac{2\lam K^{1-\beta_2}(\beta_2-1)}{\kappa^2(\beta_2-r_1)(\beta_2-r_2)}f_2(x,h)^{\beta_2-1}\\
&+ \frac{2(\beta_1-1)h^{\alpha\gamma_1}}{\kappa^2(\beta_1-r_1)(\beta_1-r_2)}f_2(x,h)^{\beta_1-1},
\end{aligned}
& \mbox{if } x_\mathrm{low}(h) \leq x \leq x_\mathrm{aggr}(h),\\
\begin{aligned}
&\frac{2(r+\lam)}{\kappa^2}C_5(h)f_3(x,h)^{r_1-1} +  \frac{2(r+\lam)}{\kappa^2}C_6(h)f_3(x,h)^{r_2-1} \\
&+ \frac{2\lam K^{1-\beta_2}(\beta_2-1)}{\kappa^2(\beta_2-r_1)(\beta_2-r_2)}f_3(x,h)^{\beta_2-1} ,
\end{aligned}
& \mbox{if } x_\mathrm{aggr}(h) < x \leq x_\mathrm{lavs}(h),
\end{cases}
\end{aligned}
\end{equation}
and the optimal bequest
\begin{equation}\label{eq: b_beta<1}
b^*(x,h) = \begin{cases}
\bigg(\frac{f_1(x,h)}{K}\bigg)^{\frac{1}{\gamma_2-1}}, &\mbox{if } x_\mathrm{bound} \leq x < x_\mathrm{low}(h),\\
\bigg(\frac{f_2(x,h)}{K}\bigg)^{\frac{1}{\gamma_2-1}}, &\mbox{if } x_\mathrm{low} \leq x < x_\mathrm{aggr}(h),\\
\bigg(\frac{f_3(x,h)}{K}\bigg)^{\frac{1}{\gamma_2-1}}, &\mbox{if } x_\mathrm{aggr} < x \leq x_\mathrm{lavs}(h).
\end{cases}
\end{equation}
Moreover, for any initial value $(X_0^*, H_0^*) = (x,h) \in \mathcal{C}$, the stochastic differential equation
\begin{equation}\label{eq: SDE}
dX_t^* = (r+\lam)X_t^*dt + \pi^*(\mu-r)dt- c^*dt - \lam b_t^*dt + \pi^*\sigma dW_t,
\end{equation}
has a unique strong solution under the optimal feedback control $(c^*, \pi^*)$.
\end{theorem}
\begin{proof}
The proof of Theorem \ref{cor: beta<1} is trivial under results of Theorem \ref{thm: beta<1} and the inverse Legendre transfrom.
{\red{}Moreover, the existence and uniqueness of the strong solution to SDE \eqref{eq: SDE} follows the same argument in the proof of Proposition 5.1 of \cite{DengLiPY2021FS}.}
\end{proof}

Based on Theorem \ref{cor: beta<1}, we can derive some asymptotic results of the optimal consumption-wealth ratio $c_t^*/X_t^*$ and the invest fraction $\pi_t^*/X_t^*$ when the wealth is sufficiently large.
{\red{}As wealth $x\rightarrow+\infty$, the running maximum $h$ updates to $h = x_\mathrm{lavs}^{-1}(x)$ and also tends to infinity.
Therefore, from the constraint that $x\leq x_\mathrm{lavs}(h)$, the asymptotic properties of optimal controls as $x\rightarrow+\infty$ should be restrained along the boundary curve $x = x_\mathrm{lavs}(h)$ as $h\rightarrow+\infty$.}
\begin{corollary}\label{cor: asy_infty_wealth}
Two limits
$
\lim\limits_{h\rightarrow+\infty} \frac{c^*(x_\mathrm{lavs}(h),h)}{x_\mathrm{lavs}(h)}
$
and
$
\lim\limits_{h\rightarrow+\infty}\frac{\pi^*(x_\mathrm{lavs}(h),h)}{x_\mathrm{lavs}(h)}
$ exist and are both positive.
Meanwhile, the asymptotic behavior of optimal bequest
$
\lim\limits_{h\rightarrow+\infty}\frac{b^*(x_\mathrm{lavs}(h),h)}{x_\mathrm{lavs}(h)}
$ also exists, and is positive if and only if $\gamma_2=(1-\alpha)\gamma_1$.
\end{corollary}
\begin{proof}
See Section \ref{proof: asy_infty_wealth}.
\end{proof}

\begin{remark}\label{remark: nonlinear}
Contrary to \cite{GuasoniHubermanR2020MFF}, all boundary curves $x_\mathrm{low}(h), x_\mathrm{aggr}(h)$ and $x_\mathrm{lavs}(h)$ in the present paper are all nonlinear functions of $h$, because the expected bequest and the optimal life insurance control are considered in our problem.
If $\lam=0$ such that there is no life insurance control, the boundary curves will become linear functions of the reference variable $h$, and the results are similar to those in \cite{GuasoniHubermanR2020MFF}.
\end{remark}

\begin{remark}
Under the optimal control $(c^*, \pi^*, b^*)$, the wealth process $X_t^*$ satisfies the constraint that $X_t^* \geq \frac{\nu H_t^*}{r+\lam}$ if the initial condition $X_0^*=x \geq \frac{\nu h}{r+\lam}$ is satisfied. Indeed, let $Z_t^* := X_t^* - \frac{\nu H_t^*}{r+\lam}$.
If $Z_t^*=0$ at some $t\geq0$, the optimal feedback controls satisfy that $c_t^* = \nu H_t^*$, $\pi_t^* = 0$, and $b_t^*=0$, indicating that $Z_s^* = 0$ and $c_s^* = \nu H_s^*$ for all $s\geq t$. That is, the optimal wealth $X_t^*$ stays at the level $\frac{\nu H_t^*}{r+\lam}$ once this level is hit.
\end{remark}

\begin{remark}
As wealth $x$ tends to the lower bound {\red$\frac{\nu h}{r+\lam}$}, the optimal bequest $b^*\rightarrow0$, and thus the optimal premium $p^* = \lam(b^*-x) <0 $.
If the parameters satisfy $\kappa^2(\beta_2^2-1)\geq2r$, the optimal premium shall always be negative. If the parameters satisfy $\kappa^2(\beta_2^2-1)<2r$, the optimal premium would be positive if $x>x^*$, where $x^*$ satisfies $f(x^*,h) < h^{(1-\alpha)\gamma_1-1}$ and
$$
\frac{\kappa^2(\beta_2^2-1)-2r}{\kappa^2(\beta_2-r_1)(\beta_2-r_2)K^{\beta_2-1}} f(x^*,h)^{\beta_2-1} = -r_1C_5(h)f(x^*,h)^{r_1-1} - r_2C_6(h)f(x^*,h)^{r_2-1} + \frac{h}{r+\lam}.
$$
\end{remark}

\section{Numerical Illustrations}\label{sec: numerical}
In this section, we numerically illustrate some quantitative properties of the feedback functions of optimal consumption, investment, and life insurance premium policy established in Theorem \ref{cor: beta<1}.
Let us choose the following values of the model parameters: $r=0.05$, $\mu = 0.1$, $\sigma=0.25$, $\rho=0.05$, $\lam=0.03$, $\nu=0.2$, $\gamma_1=0.5$, $\gamma_2=0.1$, $\alpha=0.7$, $K=5$, and reference level $h=1$.
In the following figures, we only change the value of one parameter (while keeping other parameters fixed) to show some sensitivity results with respect to that parameter.

The left panel of Figure \ref{pic: boundary_curves} shows that three boundary curves $x_\mathrm{low}(h)$, $x_\mathrm{aggr}(h)$, and $x_\mathrm{lavs}(h)$ are increasing nonlinear functions of $h$.
The graphs are consistent with the intuition that if the past reference level is higher, the investor would expect larger wealth thresholds to trigger the change of consumption from the low constraint $c=\nu h$ to $c>\nu h$, and from $c < h$ to the historical maximum $c=h$, respectively.
{\red{}From the middle panel, the higher mortality probability motivates the agent to reduce all thresholds and consume more aggressively before the death occurs. It can be seen from the right panel of Figure \ref{pic: boundary_curves} that $x_\mathrm{low}$, $x_\mathrm{aggr}$ and $x_\mathrm{lavs}$ are all decreasing in the shortfall aversion parameter $\alpha$, indicating that the more shortfall averse the agent is, the more eager the agent is to consume at the historical maximum level by lowering the corresponding thresholds.}

\begin{figure}[htbp]
\centering
\includegraphics[width=\textwidth]{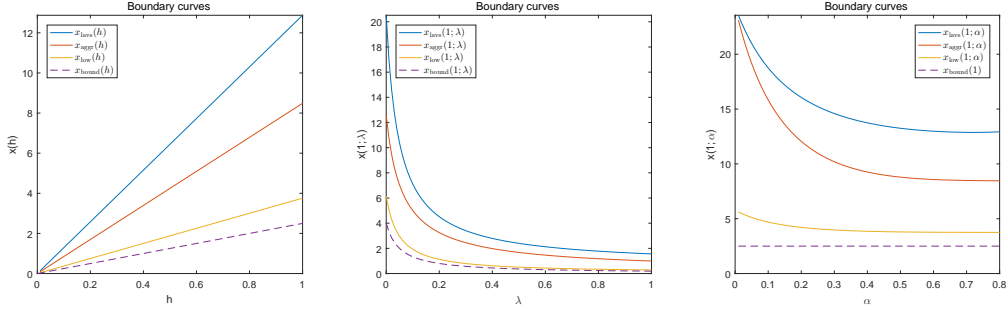}
\vspace{-0.2in}
\caption{{\small Boundary curves $x_\mathrm{bound}$, $x_\mathrm{low}$, $x_\mathrm{aggr}$ and $x_\mathrm{lavs}$ with respect to the reference variable $h$ (left), the force of mortality $\lam$ (middle), and the shortfall aversion parameter $\alpha$ (right), respectively.}}
\label{pic: boundary_curves}
\end{figure}

{\red{}Figures \ref{pic: se_lam} to \ref{pic: se_K} show the sensitivity results of optimal controls on the force of mortality $\lam$, the shortfall aversion $\alpha$ and the bequest motive $K$, respectively. From Figure \ref{pic: se_lam}, when the wealth level $x$ is sufficiently large, the higher force of mortality motivates the larger optimal consumption and higher optimal insurance premium but results in the lower portfolio allocation in the risky asset. These observations can be explained by the real life situation that the agent would spare more cash from the financial market to consume more and purchase more life insurance in view of the higher probability of death. It is interesting to see from Figure \ref{pic: se_alpha} that a larger shortfall aversion parameter $\alpha$ (i.e., the stronger desire to consume at the historical peak level), leads to a larger optimal insurance premium, which is similar to the observation made in Ben-Arab et al. \cite{Ben} that higher consumption habits would increase the demand of life insurance. It is also consistent with two real life observations: $(i)$ the agent who develops higher standard of living due to larger $\alpha$ would purchase more life insurance, possibly to ensure that the left family members can afford the high living standard after the death of the agent; $(ii)$ when the agent has sufficient wealth, purchasing more life insurance can also be an effective instrument to reduce some spared cash and smooth out the consumption path so that the reference level will not increase significantly. From Figure \ref{pic: se_K}, it is natural to see that the higher bequest motive $K$ yields higher demand of life insurance and lower portfolio allocation. We stress that a higher bequest goal also lowers all consumption thresholds and increases the consumption level. This can be explained by the real life observation that the agent who cares more about the life insurance protection is more likely to develop a higher standard of living and consume more aggressively due to a higher reference level.}

\begin{figure}[htbp]
\centering
\includegraphics[width=\textwidth]{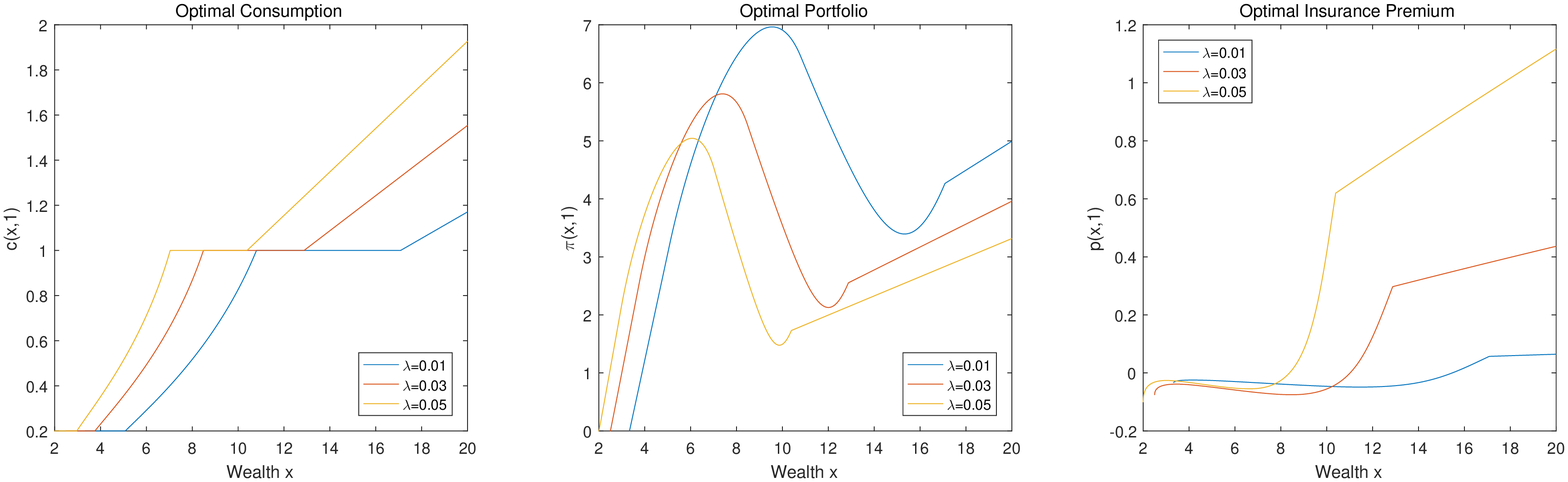}
\vspace{-0.2in}
\caption{{\small Optimal consumption, portfolio and insurance premium for various force of mortality.}}
\label{pic: se_lam}
\end{figure}

\begin{figure}[htbp]
\centering
\includegraphics[width=\textwidth]{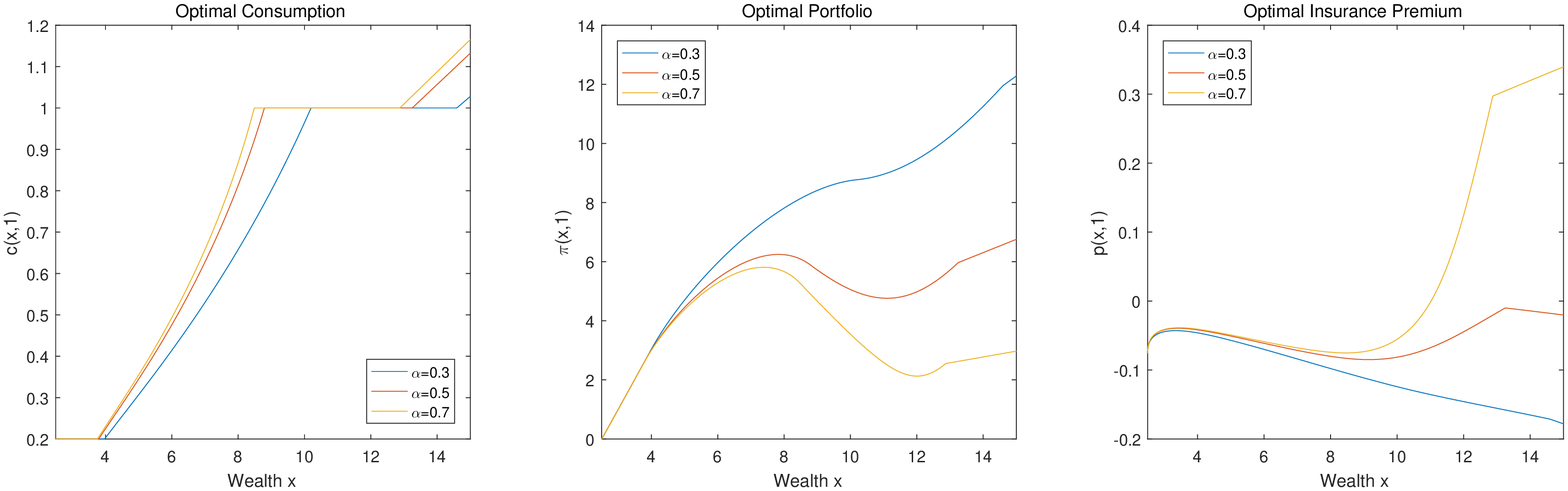}
\vspace{-0.2in}
\caption{{\small Optimal consumption, portfolio and insurance premium for various shortfall aversion.}}
\label{pic: se_alpha}
\end{figure}

\begin{figure}[htbp]
\centering
\includegraphics[width=\textwidth]{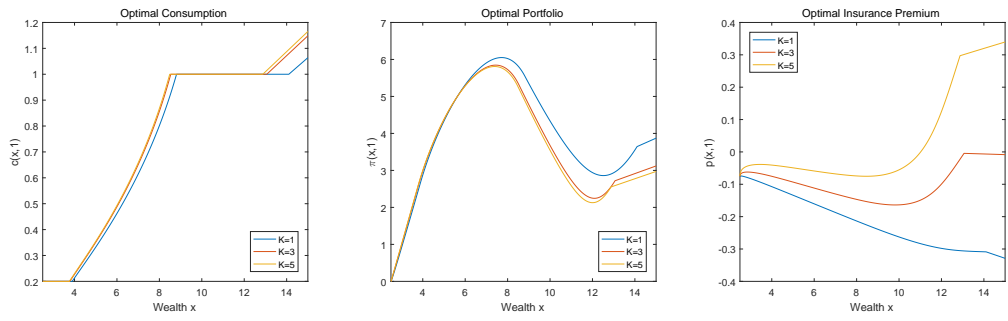}
\vspace{-0.2in}
\caption{{\small Optimal consumption, portfolio and insurance premium for various bequest motives.}}
\label{pic: se_K}
\end{figure}

{\red{}Figure \ref{pic: se_nu} shows sensitivity results of optimal controls on the drawdown constraint parameter $\nu$. When the wealth level is sufficient such that the drawdown constraint on consumption rate can be supported, the larger parameter $\nu$ increases all thresholds for consumption plan and also leads to higher past spending maximum when the wealth level is large. Due to the higher minimum consumption rate at the drawdown constraint level and higher consumption when wealth level is large, it is reasonable to observe that the larger parameter $\nu$ reduces the incentives of portfolio allocation and life insurance when the wealth is sufficient.}

\begin{figure}[htbp]
\centering
\includegraphics[width=\textwidth]{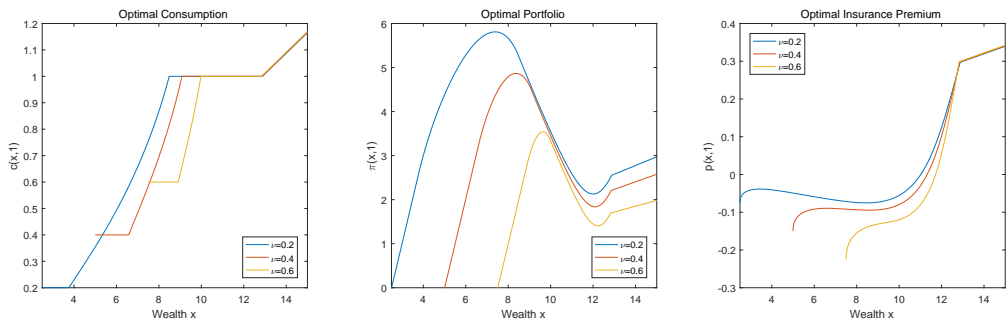}
\vspace{-0.2in}
\caption{{\small Optimal consumption, portfolio and insurance premium for various drawdown constraint parameters.}}
\label{pic: se_nu}
\end{figure}

%

Figures \ref{pic: simu_1_x_5} and \ref{pic: simu_2_x_5} present the simulated paths of the optimal wealth, the optimal consumption, the optimal portfolio, and the optimal life insurance premium in ten years in three different models:
1) our proposed model with life insurance, reference to past spending maximum and drawdown constraint (our model);
{\red2) the shortfall aversion model in \cite{GuasoniHubermanR2020MFF};
3) the standard optimal consumption and life insurance model.}
{\red{}If we does not consider life insurance and drawdown constraint, that is, $\lambda=0$ and $\nu=0$, our model is equivalent to Guasoni's model (Guasoni).
Moreover, a non-habit individual would not be affected by the consumption path in her model and can be characterized by our model if shortfall aversion $\alpha=0$ (non-habit).
We set the initial wealth to be $X_0=3.5$. One can observe that the optimal wealth sample path in the non-habit model dominates other two counterparts, and the optimal wealth path in our model grows slowest due to the life insurance purchase and the consumption reference.
For the same reasoning, the portfolio allocation in our model is also the least.
Regarding the optimal consumption paths, the simulated path in our model is smoother than other two and the overall consumption level is also highest due to the drawdown constraint.
Finally, comparing the demand of life insurance between our model and the non-habit model, our life insurance premium path becomes much smoother, indicating that the reference to past consumption not only leads to the stable consumption behavior, but also helps to smooth out the optimal premium plan.}

\begin{figure}[htbp]
\centering
\includegraphics[width=5.5in]{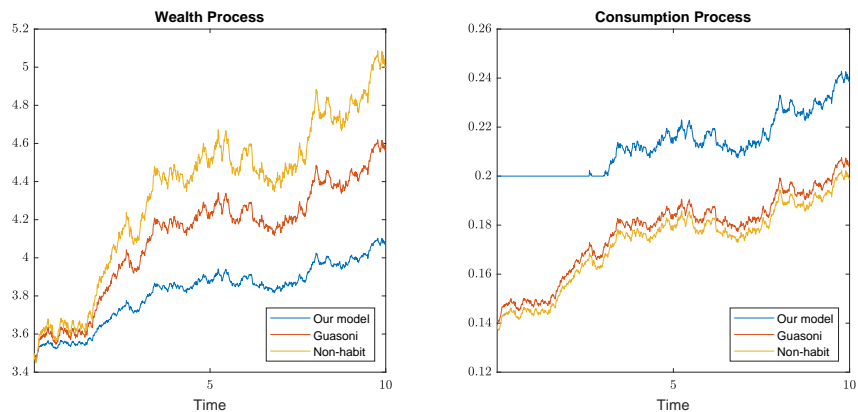}
\vspace{-0.2in}
\caption{{\small Wealth and consumption processes ($X_0=3.5$).}}
\label{pic: simu_1_x_5}
\end{figure}
\begin{figure}[htbp]
\centering
\includegraphics[width=5.5in]{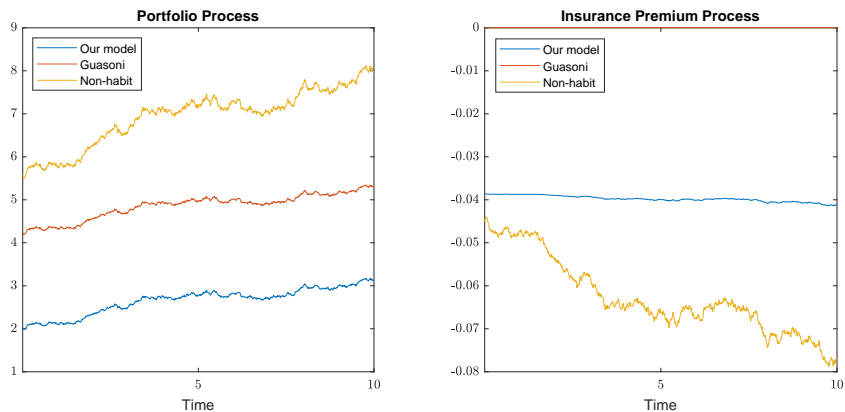}
\vspace{-0.2in}
\caption{{\small Portfolio and life insurance premium processes ($X_0=3.5$).}}
\label{pic: simu_2_x_5}
\end{figure}

{\red\section{Conclusion}\label{sec: conclusion}
In this paper, we study the optimal life insurance problem together with dynamic portfolio and consumption plans in a new framework under the shortfall aversion preference and a drawdown constraint on consumption. For the infinite horizon stochastic control problem, we can find the classical solution to the associated HJB equation in piecewise form and characterize the optimal feedback controls explicitly across different regions. Thanks to our analytical results, we can numerically illustrate the sensitivity results of the optimal strategies on model parameters and conclude some interesting financial implications.

Several directions of future research can be considered. For example, one may consider the problem in the market model with regime switching, and discuss some quantitative changes in optimal strategies in the bull and bear regime states. It will also be appealing to study the more challenging problem over a finite horizon, in which the analytical characterization of the value function is not promising and all boundary curves to distinguish different optimal feedback controls will be time-dependent. Some new techniques are needed to tackle the parabolic PDE problem.}

\section{Proofs}\label{sec: proof}

\subsection{Proof of the verification theorem}
The proof of the theorem is based on some auxiliary results. We first present some asymptotic results on the coefficients in Proposition \ref{prop: dual_solution_beta<1}, whose proof is straightforward and hence omitted.

\begin{remark}\label{lemma: C_order}
Based on the semi-analytical forms in Proposition \eqref{prop: dual_solution_beta<1}, we note that
$$\begin{aligned}
C_2(h) &= O(h^{r_1(1-\alpha)\gamma_1+r_2}), ~
C_3(h) = O(h^{r_2(1-\alpha)\gamma_1+r_1}), ~
C_4(h) = O(h^{r_1(1-\alpha)\gamma_1+r_2}), \\
C_5(h) &= O(h^{r_2(1-\alpha)\gamma_1+r_1}), ~
C_6(h) = O(h^{r_1(1-\alpha)\gamma_1+r_2}),
\end{aligned}$$
as $h\rightarrow+\infty$, which will be used in later proofs.
\end{remark}


By following similar proofs of Lemma 5.1 to Lemma 5.3 in Deng et al. \cite{DengLiPY2021FS} and using asymptotic results in Remark \ref{lemma: C_order}, we can readily obtain the next three lemmas.
\begin{lemma}\label{lemma: 1}
For any $y>0$ and $h\geq0$, the dual transform $v(y,h)$ of value function $u(x,h)$ satisfies
$$
v(y,h) = \mathbb{E}\bigg[ \int_0^\infty e^{-rt} [\tV(Y_t(y), \hat{H}_t(y))+\bar{V}(Y_t(y))]dt\bigg],
$$
where $Y_t(\cdot)$ and $\hat{H}_t(\cdot)$ are defined in \eqref{eq: H_hat}.
\end{lemma}

\begin{lemma}\label{lemma: 2}
For all $y>0$, we have $H_t^\dag = \hat{H}_t(y)$, $t\geq0$, and hence
$$
\mathbb{E}\bigg[ \int_0^\infty e^{-rt} \tV(Y_t(y), H_t^\dag(y))dt \bigg]
= \mathbb{E}\bigg[ \int_0^\infty e^{-rt} \tV(Y_t(y), \hat{H}_t(y))dt \bigg].
$$
\end{lemma}
%

\begin{lemma}\label{lemma: 3}
The inequality in \eqref{eq: verification_main} becomes equality with $c_t^* = c^\dag(Y_t(y^*), \hat{H}_t(y^*))$ and $b_t^* = b^\dag(Y_t(y^*), \hat{H}_t(y^*))$, $t\geq0$, with $y^* = y^*(x,h)$ as the unique solution to
$$\mathbb{E}[\int_0^\infty (c^\dag(Y_t(y^*), \hat{H}_t(y^*)) + \lam b^\dag(Y_t(y^*), \hat{H}_t(y^*)))M_tdt] = x.$$
\end{lemma}
Let us continue to prove some other auxiliry results.
\begin{lemma}\label{lemma: 4}
The following transversality condition holds that for all $y>0$,
$$
\lim\limits_{T\rightarrow+\infty} \mathbb{E}\bigg[ e^{-rT}v(Y_T(y), \hat{H}_T(y)) \bigg] = 0.
$$
\end{lemma}

\begin{proof}
Let us recall that
$$
\hat{H}_t(y) := h\vee \bigg( (1-\alpha)^{-\frac{1}{(1-\alpha)\gamma_1-1}} (\inf_{s\leq t}Y_s(y))^{\frac{1}{(1-\alpha)\gamma_1-1}} \bigg).
$$

Let us firstly consider the case $c_T = 0$.
We first write that
\begin{equation}\label{eq: trans_R1}
\begin{aligned}
e^{-rT}\E[v(Y_T(y),\hat{H}_T(y))] = e^{-rT}\E\bigg[C_2(\hat{H}_T(y)) Y_T(y)^{r_2} + \frac{2\lam K^{1-\beta_2}}{\kappa^2\beta_2(\beta_2-r_1)(\beta_2-r_2)}Y_T(y)^{\beta_2}  \\
+ \frac{\nu^\gamma}{(r+\lam)\gamma_1}\hat{H}_T(y)^{(1-\alpha)\gamma} - \frac{\nu \hat{H}_T(y)}{r+\lam} Y_T(y)\bigg],
\end{aligned}
\end{equation}
{\red where the last two terms can vanish due to  Lemma \ref{lemma: claim1} and Lemma \ref{lemma: claim3} respectively, and the last third term can also vanish because of Lemma \ref{lemma: claim2} and the fact $\beta_2>r_2$ by \textbf{Assumption (A1)}.}
For the first term in \eqref{eq: trans_R1}, since $Y_T(y) > \hat{H}_T(y)^{(1-\alpha)\gamma_1-1}$, we have
$$
\begin{aligned}
e^{-rT}\E\bigg[ C_2(\hat{H}_T(y)) (Y_T(y))^{r_2} \bigg]
&= O\big(e^{-rT}C_2(\hat{H}_T(y)) \hat{H}_T(y)^{r_2((1-\alpha)\gamma_1-1)} \big) \\
&= O\big(e^{-rT}\hat{H}_T(y)^{r_1(1-\alpha)\gamma_1+r_2} \hat{H}_T(y)^{r_2((1-\alpha)\gamma_1-1)} \big) \\
&= O\big(e^{-rT}\hat{H}_T(y)^{(1-\alpha)\gamma_1}\big),
\end{aligned}
$$
which vanishes as $T\rightarrow+\infty$ due to Lemma \ref{lemma: claim1}.

We then consider the case $0<c_T<\hat{H}_T(y)$.
\begin{equation}\label{eq: trans_R2}\begin{aligned}
&\E[e^{-rT}v(Y_T(y),\hat{H}_T(y))] \\
= &e^{-rT}\E\bigg[
C_3(\hat{H}_T(y)) Y_T(y)^{r_1} + C_4(\hat{H}_T(y)) Y_T(y)^{r_2}\\
&+ \frac{2\lam K^{1-\beta_2}}{\kappa^2\beta_2(\beta_2-r_1)(\beta_2-r_2)}Y_T(y)^{\beta_2}+ \frac{2\hat{H}_T(y)^{\alpha\beta_1}}{\kappa^2\beta_1(\beta_1-r_1)(\beta_1-r_2)}Y_T(y)^{\beta_1} \bigg].
\end{aligned}\end{equation}
We consider asymptotic behavior of the above equation term by term as $T\rightarrow+\infty$.

{\red Thanks to \textbf{Assumption (A1)}, $\beta_2>r_2$,} the third term can vanish due to Lemma \ref{lemma: claim2}.
For the fourth term in \eqref{eq: trans_R2}, since $Y_T(y) \geq \hat{H}_T(y)^{(1-\alpha)\gamma_1-1}$ and $\beta_1 = \frac{\gamma_1}{\gamma_1-1}<0$, we have
$$\begin{aligned}
\E[e^{-rT}\hat{H}_T(y)^{\alpha\beta_1}Y_T(y)^{\beta_1}]
= {\red O\big(e^{-rT}\E\big[\hat{H}_T(y)^{\alpha\beta_1+\beta_1((1-\alpha)\gamma_1-1)}\big]\big)
= O\big(e^{-rT}\E\big[\hat{H}_T(y)^{(1-\alpha)\gamma_1}\big]\big)},
\end{aligned}$$
which also vanishes as $T\rightarrow+\infty$ due to {\red Lemma \ref{lemma: claim1}}.

Let us consider the terms containing $C_3(\hat{H}_T(y))$ and $C_4(\hat{H}_T(y))$ in equation \eqref{eq: trans_R2}.
Because of the constraint $Y_t(y) = O(\hat{H}_T(y)^{(1-\alpha)\gamma_1-1})$, we can deduce that
$$
\begin{aligned}
e^{-rT}\E\bigg[ C_3(\hat{H}_T(y)) (Y_T(y))^{r_1} \bigg]
&= O\big(e^{-rT}C_3(\hat{H}_T(y)) \hat{H}_T(y)^{r_1((1-\alpha)\gamma_1-1)} \big) \\
&= O\big(e^{-rT}\hat{H}_T(y)^{r_2(1-\alpha)\gamma_1+r_1} \hat{H}_T(y)^{r_1((1-\alpha)\gamma_1-1)} \big) \\
&= O\big(e^{-rT}\hat{H}_T(y)^{(1-\alpha)\gamma_1}\big),
\end{aligned}
$$
which converges to 0 by Lemma \ref{lemma: claim1}. \\

In addition, since $Y_T(y) \geq \hat{H}_T(y)^{(1-\alpha)\gamma_1-1}$ and $r_2<0$, we obtain
$$
\begin{aligned}
e^{-rT}\E\bigg[ C_4(\hat{H}_T(y)) (Y_T(y))^{r_2} \bigg]
&= O\big(e^{-rT}C_4(\hat{H}_T(y)) \hat{H}_T(y)^{r_2((1-\alpha)\gamma_1-1)} \big) \\
&= O\big(e^{-rT}\hat{H}_T(y)^{r_1(1-\alpha)\gamma_1+r_2} \hat{H}_T(y)^{r_2((1-\alpha)\gamma_1-1)} \big) \\
&= O\big(e^{-rT}\hat{H}_T(y)^{(1-\alpha)\gamma_1}\big),
\end{aligned}
$$
which vanishes as $T\rightarrow+\infty$ {\red by Lemma \ref{lemma: claim1}}.

Finally, we consider the case $C_T = \hat{H}_T(y)$ and write that
\begin{equation}\label{eq: trans_R3}
\begin{aligned}
&\E[e^{-rT}v(Y_T(y), \hat{H}_T(y))] \\
= &e^{-rT}\E\bigg[ C_5(\hat{H}_T(y))Y_T(y)^{r_1}
+ C_6(\hat{H}_T(y))Y_T(y)^{r_2} \\
&+ \frac{2\lam K^{1-\beta_2}}{\kappa^2\beta_2(\beta_2-r_1)(\beta_2-r_2)}Y_T(y)^{\beta_2}+ \frac{1}{(r+\lam)\gamma_1}\hat{H}_T(y)^{(1-\alpha)\gamma_1} - \frac{\hat{H}_T(y)}{(r+\lam)}Y_T(y)\bigg],
\end{aligned}
\end{equation}
where the last three terms converge to 0 by Lemma \ref{lemma: claim2} with {\red\textbf{Assumption (A1)}}, Lemma \ref{lemma: claim1}, and Lemma \ref{lemma: claim3}, respectively.

For the first term in \eqref{eq: trans_R3}, since $Y_T(y)\leq \hat{H}_T(y)^{(1-\alpha)\gamma_1-1}$, we have
$$
\begin{aligned}
e^{-rT}\E\bigg[ C_5(\hat{H}_T(y)) (Y_T(y))^{r_1} \bigg]
&= O\big(e^{-rT}C_5(\hat{H}_T(y)) \hat{H}_T(y)^{r_1((1-\alpha)\gamma_1-1)} \big) \\
&= O\big(e^{-rT}\hat{H}_T(y)^{r_2(1-\alpha)\gamma_1+r_1} \hat{H}_T(y)^{r_1((1-\alpha)\gamma_1-1)} \big) \\
&= O\big(e^{-rT}\hat{H}_T(y)^{(1-\alpha)\gamma_1}\big),
\end{aligned}
$$
which converges to 0 as $T\rightarrow+\infty$ {\red by Lemma \ref{lemma: claim1}}. \\
For the second term in \eqref{eq: trans_R3}, by $Y_T(y)\geq (1-\alpha)\hat{H}_T(y)^{(1-\alpha)\gamma_1-1}$ and $r_2<0$, we have
$$
\begin{aligned}
e^{-rT}\E\bigg[ C_6(\hat{H}_T(y)) (Y_T(y))^{r_2} \bigg]
&= O\big(e^{-rT}C_6(\hat{H}_T(y)) \hat{H}_T(y)^{r_2((1-\alpha)\gamma_1-1)} \big) \\
&= O\big(e^{-rT}\hat{H}_T(y)^{r_1(1-\alpha)\gamma_1+r_2} \hat{H}_T(y)^{r_2((1-\alpha)\gamma_1-1)} \big) \\
&= O\big(e^{-rT}\hat{H}_T(y)^{(1-\alpha)\gamma_1}\big),
\end{aligned}
$$
which also vanishes as $T\rightarrow+\infty$ {\red by Lemma \ref{lemma: claim1}}.
Therefore, we get the desired result.
\end{proof}

\begin{lemma}\label{lemma: 4_2}
For any $T>0$, we have
$$
\lim\limits_{n\rightarrow+\infty} \E\big[e^{-r\tau_n}v(Y_{\tau_n}(y), \hat{H}_{\tau_n}(y)) \mathbf{1}_{\{T> \tau_n\}} \big] = 0,
$$
where $\tau_n$ is defined by
$$
\tau_n = \inf\{t\geq0 | Y_t(y)\geq n, \hat{H}_t(y)\geq \big((1-\alpha)n\big)^{-\frac{1}{(1-\alpha)\gamma_1-1}} \}.
$$
\end{lemma}
\begin{proof}
By the definition of $\tau_n$, for all $t\leq\tau_n$, we have $Y_t(y) \in \big[\frac1n, n\big]$, and thus
$$h\leq \hat{H}_t(y) \leq \max(h, \big((1-\alpha)n\big)^{-\frac{1}{(1-\alpha)\gamma_1-1}}) = O(1) + O(n^{-\frac{1}{(1-\alpha)\gamma_1-1}}). $$
Therefore, we have $Y_t(y)^{r_1} \leq n^{r_1}$, $Y_t(y)^{r_2} \leq \big(\frac1n \big)^{r_2}= n^{-r_2}$.
Then we shall obtain the order of $v(Y_{\tau_n}(y), \hat{H}_{\tau_n}(y))$ in three cases, in the sense that $c_{\tau_n}^*=0$, $0<c_{\tau_n}^*<\hat{H}_{\tau_n}(y)$, and $c_{\tau_n}^* = \hat{H}_{\tau_n}(y)$.

Similar to the proof of Lemma \ref{lemma: 4}, if $c_{\tau_n}^*=0$, we have
$$
\begin{aligned}
v(Y_{\tau_n}(y), \hat{H}_{\tau_n}(y))
&= C_2(\hat{H}_{\tau_n}(y))Y_{\tau_n}(y)^{r_2} + \frac{2\lam K^{1-\beta_2}}{\kappa^2\beta_2(\beta_2-r_1)(\beta_2-r_2)}Y_{\tau_n}(y)^{\beta_2}  \\
&~~+\frac{\nu^\gamma}{(r+\lam)\gamma_1}\hat{H}_{\tau_n}(y)^{(1-\alpha)\gamma_1} - \frac{\nu \hat{H}_{\tau_n}(y)}{r+\lam} Y_{\tau_n}(y) \\
&= O(n^{-r_2}) + O(n^{-\beta_2}) + O(n^{-\frac{(1-\alpha)\gamma_1}{(1-\alpha)\gamma_1-1}}) + O(n^{\frac{(1-\alpha)\gamma_1-2}{(1-\alpha)\gamma_1-1}}) \\
&= O(n^{r^*}),
\end{aligned}
$$
where $r^*:= \max\big\{-r_2, {\red-\beta_1,} \frac{(1-\alpha)\gamma_1-2}{(1-\alpha)\gamma_1-1}, \frac{(r_1-r_2)(1-\alpha)\gamma_1-2r_1}{(\alpha-1)\gamma_1-1}\big\}$.
{\red Here, we have $-\beta_2<-r_2$ by \textbf{Assumption (A1)}. }
If $0<c_{\tau_n}^*<\hat{H}_{\tau_n}(y)$, we have
$$
\begin{aligned}
v(Y_{\tau_n}(y),\hat{H}_{\tau_n}(y)) = &
C_3(\hat{H}_{\tau_n}(y)) Y_{\tau_n}(y)^{r_1} + C_4(\hat{H}_{\tau_n}(y)) Y_{\tau_n}(y)^{r_2}\\
&~~+ \frac{2\lam K^{1-\beta_2}}{\kappa^2\beta_2(\beta_2-r_1)(\beta_2-r_2)}Y_{\tau_n}(y)^{\beta_2} +\frac{2\hat{H}_{\tau_n}(y)^{\alpha\beta_1}}{\kappa^2\beta_1(\beta_1-r_1)(\beta_1-r_2)}Y_{\tau_n}(y)^{\beta_1} \\
&= O(n^{\frac{(r_1-r_2)(1-\alpha)\gamma_1-2r_1}{(\alpha-1)\gamma_1-1}}) + O(n^{-r_2}) + O(n^{-\beta_2}) + O(n^{-\beta_1}) \\
&= O(n^{r^*}).
\end{aligned}
$$
If $c_{\tau_n} = \hat{H}_{\tau_n}(y)$, we have
$$\begin{aligned}
v(Y_{\tau_n}(y), \hat{H}_{\tau_n}(y))
&=  C_5(\hat{H}_{\tau_n}(y))Y_{\tau_n}(y)^{r_1} + C_6(\hat{H}_{\tau_n}(y))Y_{\tau_n}(y)^{r_2} \\
&~~+ \frac{2\lam K^{1-\beta_2}}{\kappa^2\beta_2(\beta_2-r_1)(\beta_2-r_2)}Y_{\tau_n}(y)^{\beta_2}+ \frac{1}{(r+\lam)\gamma_1}\hat{H}_{\tau_n}(y)^{(1-\alpha)\gamma_1} - \frac{\hat{H}_{\tau_n}(y)}{r+\lam}Y_T(y)\\
&= O(n^{\frac{(r_1-r_2)(1-\alpha)\gamma_1-2r_1}{(\alpha-1)\gamma_1-1}}) + O(n^{-r_2}) + O(n^{-\beta_2}) + O(n^{\frac{-(1-\alpha)\gamma_1}{(1-\alpha)\gamma_1-1}}) + O(n^{\frac{(1-\alpha)\gamma_1-2}{(1-\alpha)\gamma_1-1}}) \\
&= O(n^{r^*}).
\end{aligned}$$
Therefore, in all the cases, $v(Y_{\tau_n}(y), \hat{H}_{\tau_n}(y)) = O(n^{r^*})$.
In addition, similar to the proof of (A.25) in \cite{GuasoniHubermanR2020MFF}, there exists some constant $C$ such that
$$
\E[\mathbf{1}_{\{\tau_n \leq T\}}] \leq n^{-2\xi}(1+y^{2\xi})e^{CT},
$$
for any $\xi\geq 1$.
Putting all the pieces together, the desired claim holds that
$$
\lim\limits_{n\rightarrow+\infty} \E\big[e^{-r\tau_n}v(Y_{\tau_n}(y), \hat{H}_{\tau_n}(y)) \mathbf{1}_{\{T> \tau_n\}} \big] = 0.
$$
\end{proof}

\begin{lemma}\label{lemma: claim1}
For $\gamma_1$ that satisfies \textbf{Assumption (A1)}, we have
\begin{equation}
\lim\limits_{T\rightarrow+\infty}\E\bigg[e^{-rT}\hat{H}_T(y)^{\gamma_1^*}\bigg] = 0,
\end{equation}
where $\gamma_1^* := (1-\alpha)\gamma_1$.
\end{lemma}
\begin{proof}
Let $\beta_1^* := \frac{\gamma_1^*}{\gamma_1^*-1} $.
It is obvious that
$$\begin{aligned}
e^{-rT}\E\bigg[\hat{H}_T(y)^{(1-\alpha)\gamma_1}\bigg]
&\leq e^{-rT}\E\bigg[ \sup\limits_{s\leq T}(1-\alpha)^{-\frac{(1-\alpha)\gamma_1}{(1-\alpha)\gamma_1-1}} Y_s(y)^{\frac{(1-\alpha)\gamma_1}{(1-\alpha)\gamma_1-1}} \bigg] + e^{-rT}\E[h^{(1-\alpha)\gamma_1} ] \\
&= e^{-rT}\E\bigg[ \sup\limits_{s\leq T}(1-\alpha)^{-\beta_1^*} Y_s(y)^{\beta_1^*} \bigg] + e^{-rT}\E[h^{(1-\alpha)\gamma_1} ], \\
\end{aligned}$$
in which it is clear that $e^{-rT}\E[h^{(1-\alpha)\gamma_1} ] = O(e^{-rT})$ as $T\rightarrow+\infty$.

Then we consider the first term $e^{-rT}\E[ \sup_{s\leq T}(1-\alpha)^{-\beta_1^*} Y_s(y)^{\beta_1^*} ]$.
Define $W_t^{(\frac12\kappa)} = W_t + \frac12\kappa t$, which is also a Brownian motion under the equivalent measure $\mathbb{Q}$, with its running maximum $(W_t^{(\frac12\kappa)})^*$.
It follows that
$$
\begin{aligned}
&e^{-rT}\E\bigg[\hat{H}_T(y)^{(1-\alpha)\gamma_1}\bigg]
\leq e^{-rT}\E\bigg[ \sup\limits_{s\leq T}(1-\alpha)^{-\beta_1^*} Y_s(y)^{\beta_1^*} \bigg] \\
=& e^{-rT}O\bigg( \E\bigg[\exp\bigg\{-\beta_1^*\sup\limits_{s\leq T}\big(\frac12\kappa^2s + \kappa W_s\big) \bigg\} \bigg] \bigg) \\
=& e^{-rT}O\bigg( \E\bigg[ \exp\bigg\{-\kappa\beta_1^* \sup\limits_{s\leq T}W_s^{(\frac12\kappa)} \bigg\} \bigg] \bigg) \\
:=& e^{-rT}O\bigg( \E\bigg[ \exp\bigg\{ aW_T^{(\zeta)} + b\bigg( W_T^{(\zeta)}\bigg)^* \bigg\} \mathbb{I}{\bigg\{\bigg( W_T^{(\zeta)}\bigg)^* \geq k\bigg\}} 
\bigg] \bigg),
\end{aligned}
$$
where $a=0$, $b = -\beta_1^*\kappa > 0$, $\zeta = \frac12\kappa > 0$, and $k = 0$.
Note that $2a+b+2\zeta > 2a+b+\zeta > 0$, thanks to the Corollary A.7 in \cite{GuasoniHubermanR2020MFF}, we have
$$\begin{aligned}
&\E\bigg[ \exp\bigg\{ aW_T^{(\zeta)} + b\bigg( W_T^{(\zeta)}\bigg)^*  \bigg\} 
\mathbb{I}{\bigg\{\bigg( W_T^{(\zeta)}\bigg)^* \geq k\bigg\}}
\bigg]
\\
=& \frac{2(a+b+\zeta)}{2a+b+2\zeta}\exp\bigg\{ \frac{(a+b)(a+b+2\zeta)}{2}T \bigg\} \Phi\bigg((a+b+\zeta)\sqrt T - \frac{k}{\sqrt T} \bigg) \\
+& \frac{2(a+\zeta)}{2a+b+2\zeta}\exp\bigg\{ (2a+b+2\zeta)k + \frac{a(a+2\zeta)}{2}T \bigg\} \Phi\bigg( -(a+\zeta)\sqrt T - \frac{k}{\sqrt T} \bigg),
\end{aligned}$$
and thus
$$\begin{aligned}
&\lim\limits_{T\rightarrow+\infty}\frac1T\log \E\bigg[ \exp\bigg\{ aW_T^{(\zeta)} + b\bigg( W_T^{(\zeta)}\bigg)^*  \bigg\} 
\mathbb{I}{\bigg\{\bigg( W_T^{(\zeta)}\bigg)^* \geq k\bigg\}}
\bigg] -r \\
=& \frac{(a+b)(a+b+2\xi)}{2} -r = \frac{\kappa^2}{2}\beta_1^*(\beta_1^*-1) -r
  = \frac{\kappa^2}{2}(\beta_1^*-r_1)(\beta_1^*-r_2).
\end{aligned}$$
It follows that
$$
\begin{aligned}
&e^{-rT} \E\bigg[ \exp\bigg\{ aW_T^{(\zeta)} + b\bigg( W_T^{(\zeta)}\bigg)^* \mathbb{I}{\bigg\{\bigg( W_T^{(\zeta)}\bigg)^* \geq k\bigg\}} \bigg\} \\
=& \exp\bigg\{\bigg(\frac1T \log \E\bigg[ \exp\bigg\{ aW_T^{(\zeta)} + b\bigg( W_T^{(\zeta)}\bigg)^*  \bigg\} 
\mathbb{I}{\bigg\{\bigg( W_T^{(\zeta)}\bigg)^* \geq k\bigg\}}
\bigg] -r\bigg)T \bigg\} \\
=& O\bigg( \exp\bigg\{ \frac{\kappa^2}{2}(\beta_1^*-r_1)(\beta_1^*-r_2) T \bigg\} \bigg),
\end{aligned}
$$
as $T\rightarrow+\infty$.
Together with the fact that $r_2<\beta_1^*<r_1$ {\red under \textbf{Assumption (A1)}}, we have $(\beta_1^*-r_1)(\beta_1^*-r_2)<0$ and thus
$$\begin{aligned}
\E\bigg[e^{-rT}\hat{H}_T(y)^{(1-\alpha)\gamma_1}\bigg]
&= O\bigg( \exp\bigg\{ \frac{\kappa^2}{2}(\beta_1^*-r_1)(\beta_1^*-r_2) T \bigg\} \bigg) + O(e^{-rT}) \\
&= O\bigg( \exp\bigg\{ \frac{\kappa^2}{2}(\beta_1^*-r_1)(\beta_1^*-r_2) T \bigg\} \bigg),
\end{aligned}$$
which tends to 0 as $T\rightarrow+\infty$.
\end{proof}

\begin{lemma}\label{lemma: claim2}
For $r_2<\beta_0<r_1$, we have
\begin{equation}
\lim\limits_{T\rightarrow+\infty}\E\bigg[e^{-rT}Y_T(y)^{\beta_0} \bigg] = 0.
\end{equation}
\end{lemma}

\begin{proof}
In fact,
$$\begin{aligned}
\E\bigg[e^{-rT}Y_T(y)^{\beta_0} \bigg]
&= e^{-rT} \E\bigg[ (ye^{rT}\cdot e^{-(r+\frac{\kappa^2}{2})T - \kappa W_T})^{\beta_0} \bigg] \\
&= y^\gamma_1 e^{-rT} \E\big[ e^{\beta_0(-\frac{\kappa^2}{2}T - \kappa W_T)}\big] \\
&= O\bigg( e^{(\beta_0-r_1)(\beta_0-r_2)\frac{\kappa^2}{2}T} \bigg),
\end{aligned}$$
which converges to 0 in view that $r_2<\beta_0<r_1$.
\end{proof}

\begin{lemma}\label{lemma: claim3}
For $\beta_1^* = \frac{\gamma_1^*}{\gamma_1^*-1}<0$ with $\gamma_1^*:=(1-\alpha)\gamma_1$, we have
\begin{equation}
\lim\limits_{T\rightarrow+\infty}\E\bigg[e^{-rT}\hat{H}_T(y)Y_T(y) \bigg] = 0.
\end{equation}
\end{lemma}

\begin{proof}
In fact,
$$
\begin{aligned}
\E\bigg[e^{-rT}\hat{H}_T(y)Y_T(y) \bigg]
\leq \E\bigg[e^{-rT}hY_T(y) \bigg] + \E\bigg[e^{-rT}Y_T(y)\sup\limits_{s\leq T}(1-\alpha)^{-\frac{1}{(1-\alpha)\gamma_1-1}} Y_s(y)^{\frac{1}{(1-\alpha)\gamma_1-1}} \bigg],
\end{aligned}
$$
where the first term converges to 0 by Lemma \ref{lemma: claim2}.
For the second term,
$$
\begin{aligned}
&\E\bigg[e^{-rT}Y_T(y)\sup\limits_{s\leq T}(1-\alpha)^{-\frac{1}{(1-\alpha)\gamma_1-1}} Y_s(y)^{\frac{1}{(1-\alpha)\gamma_1-1}} \bigg] \\
&= O\bigg(\E\bigg[e^{-rT}Y_T(y)\sup\limits_{s\leq T} Y_s(y)^{\frac{1}{\gamma_1^*-1}} \bigg]\bigg) \\
&= e^{-rT}O\bigg(\E\bigg[\exp\bigg\{-\frac{\kappa^2}{2}T-\kappa W_T-\frac{1}{\gamma_1^*-1}\sup\limits_{s\leq T}\bigg(\frac{\kappa^2}{2}s+\kappa W_s\bigg) \bigg\}\bigg]\bigg) \\
&= e^{-rT}O\bigg(\E\bigg[\exp\bigg\{ -\kappa W^{(\zeta)}-\frac{\kappa}{\gamma_1^*-1}\bigg(W_T^{(\zeta)}\bigg)^* \bigg\}\bigg]\bigg) \\
&= e^{-rT}O\bigg(\E\bigg[\exp\bigg\{ a_1 W^{(\zeta)}+b_1\bigg(W_T^{(\zeta)}\bigg)^* \bigg\}\mathbb{I}{\bigg\{\bigg( W_T^{(\zeta)}\bigg)^* \geq k\bigg\}}\bigg]\bigg), \\
\end{aligned}
$$
where $a_1 = -\kappa$, $b_1 = -\frac{\kappa}{\gamma_1^*-1}>0$, $\zeta=\frac12\kappa$, and $k=0$.
Note that $2a_1+b_1+2\zeta = \frac{\gamma_1^*}{1-\gamma_1^*}>0$ and $a_1+\zeta < 0$, thank to the Corollary A.7 in \cite{GuasoniHubermanR2020MFF}, we can derive
$$
\begin{aligned}
&e^{-rT}O\bigg(\E\bigg[\exp\bigg\{ a_1 W^{(\zeta)}+b_1\bigg(W_T^{(\zeta)}\bigg)^* \bigg\}\mathbb{I}{\bigg\{\bigg( W_T^{(\zeta)}\bigg)^* \geq k\bigg\}}\bigg]\bigg) \\
= &O\bigg(\exp\bigg\{ \bigg(\frac{(a_1+b_1)(a_1+b_1+2\zeta)}{2} -r \bigg)T \bigg\}\bigg)
+ O\bigg(\exp\bigg\{ \bigg(\frac{a_1(a_1+2\zeta)}{2} -r \bigg)T \bigg\}\bigg),
\end{aligned}$$
where the second term equals $O(\exp\{-rT\})$ as $a_1+2\zeta=0$.
For the first term,
$$
\begin{aligned}
&\frac{(a_1+b_1)(a_1+b_1+2\zeta)}{2} -r
= \frac{\kappa^2}{2}\bigg( \frac{\gamma_1^*}{\gamma_1^*-1}\cdot\frac{1}{\gamma_1^*-1} \bigg) - r \\
=& \frac{\kappa^2}{2}\bigg( \beta_1^*(\beta_1^*-1)  -\frac{2(r+\lam)}{\kappa^2}\bigg)
= \frac{\kappa^2}{2}(\beta_1^*-r_1)(\beta_1^*-r_2).
\end{aligned}
$$
{\red Thanks to \textbf{Assumption (A1)}, we have $\beta_1^* > r_2$, and therefore $\frac{\kappa^2}{2}(\beta_1^*-r_1)(\beta_1^*-r_2) < 0$.}
In summary, we complete the proof.
\end{proof}

\subsection{Proof of Lemma \ref{lemma: beta<1}}\label{sec: proof_convex_dual}
We prove $v_{yy}(y,h)>0$ the three regions: $y>\nu^{\gamma_1-1}h^{(1-\alpha)\gamma_1-1}$, $h^{(1-\alpha)\gamma_1-1}\leq y\leq \nu^{\gamma_1-1}h^{(1-\alpha)\gamma_1-1}$, and $(1-\alpha)h^{(1-\alpha)\gamma-1}\leq y < h^{(1-\alpha)\gamma-1}$,  respectively.
To be more specific, we first analyze $v_{yy}(y,h)$ in the region $(1-\alpha)h^{(1-\alpha)\gamma-1}\leq y < h^{(1-\alpha)\gamma-1}$, then the region $h^{(1-\alpha)\gamma-1}\leq y\leq \nu^{\gamma_1-1}h^{(1-\alpha)\gamma_1-1}$, and finally the region $y>\nu^{\gamma_1-1}h^{(1-\alpha)\gamma_1-1}$.
\begin{itemize}
\item[(i)] In the region $(1-\alpha)h^{(1-\alpha)\gamma-1} \leq y < h^{(1-\alpha)\gamma-1}$, $v_{yy}(y, h) = r_1(r_1-1)C_5(h)y^{r_1-2} + r_2(r_2-1)C_6(h)y^{r_2-2} + \frac{2\lam (\beta_2-1) K^{1-\beta_2}}{\kappa^2(\beta_2-r_1)(\beta_2-r_2)}y^{\beta_2-2}$.
Since $r_1(r_1-1) = r_2(r_2-1) = \frac{2(r+\lam)}{\kappa^2} > 0$ and $\frac{2\lam (\beta_2-1) K^{1-\beta_2}}{\kappa^2(\beta_2-r_1)(\beta_2-r_2)}>0$,
we only need to prove $C_5(h) \geq 0$ and $C_6(h)>0$.
According to \eqref{eq: c_beta<1}, we can easily deduce that $C_5(h)>0$ and $C_6(h)>0$.

\item[(ii)]
In the region $h^{(1-\alpha)\gamma_1-1}\leq y\leq \nu^{\gamma_1-1}h^{(1-\alpha)\gamma_1-1}$,
because $r_1(r_1-1) = r_2(r_2-1) = \frac{2(r+\lam)}{\kappa^2}$,
we can deduce that
$$
\begin{aligned}
v_{yy}(y, h)
&= r_1(r_1-1) C_3(h)y^{r_1-2} +  r_2(r_2-1) C_4(h)y^{r_2-2}\\
&~~
+ \frac{2\lam (\beta_2-1) K^{1-\beta_2}}{\kappa^2(\beta_2-r_1)(\beta_2-r_2)}y^{\beta_2-2}
+\frac{2(\beta_1-1)h^{\alpha\beta_1}}{\kappa^2(\beta_1-r_1)(\beta_1-r_2)}y^{\beta_1-2} \\
&= \frac{2(r+\lam)}{\kappa^2}\bigg( C_3(h)y^{r_1-\beta_1} + C_4(h)y^{r_2-\beta_1} + \frac{(\beta_1-1)h^{\alpha\beta_1}}{(r+\lam)(\beta_1-r_1)(\beta_1-r_2)} \bigg)y^{\beta_1-2}\\
&~~+ \frac{2\lam (\beta_2-1) K^{1-\beta_2}}{\kappa^2(\beta_2-r_1)(\beta_2-r_2)}y^{\beta_2-2}. \\
\end{aligned}
$$
Let us define that $\varphi(y):= C_3(h)y^{r_1-\beta_1} + C_4(h)y^{r_2-\beta_1} + \frac{(\beta_1-1)h^{\alpha\beta_1}}{(r+\lam)(\beta_1-r_1)(\beta_1-r_2)}$.
Because the last term in the above equation is positive,  it is enough to verify $\varphi(y) > 0$.
We separate the proof into the following steps: (1) showing $\varphi(y)$ is either monotone or first increasing and then decreasing; (2) show $\varphi(y)>0$ at two points $y=\nu^{\gamma_1-1}h^{(1-\alpha)\gamma_1-1}$ and $y=h^{(1-\alpha)\gamma_1-1}$.

Indeed, the extreme point $y^\dag$ of $\varphi(y)$ should satisfy the first order condition $\varphi'(y^\dag) = 0$, i.e.
$$
C_3(h)(r_1-\beta_1)(y^\dag)^{r_1-\beta_1-1} + C_4(h)(r_2-\beta_1)(y^\dag)^{r_2-\beta_1-1} = 0.
$$
We remark that $C_3(h)<0$, $r_1-\beta_1>0$, while {\red $C_4(h)(r_2-\beta_1)$} can be negative or positive.
If {\red$C_4(h)(r_2-\beta_1)\leq 0$}, there is no solution for $y^\dag$, hence $\varphi(y)$ is monotone.
If {\red$C_4(h)(r_2-\beta_1)> 0$}, there exists a unique real solution to the above equation
$$
y^\dag = \bigg( \frac{(\beta_1-r_2)C_4(h)}{(r_1-\beta_1)C_3(h)} \bigg)^{\frac{1}{r_1-r_2}},
$$
which might fall into the interval $[h^{(1-\alpha)\gamma-1}, \nu^{\gamma_1-1}h^{(1-\alpha)\gamma_1-1}]$.
Noticing that $C_3(h)<0$, $(r_1-\beta_1)>0$, and
$$
\varphi'(y) = C_3(h)(r_1-\beta_1)y^{r_1-\beta_1-1} + C_4(h)(r_2-\beta_1)y^{r_2-\beta_1-1},
$$
it follows that when $y\leq y^\dag$, $\varphi'(y)\geq0$; when $y>y^\dag$, $\varphi'(y)\leq0$.
Hence $\varphi(y)$ increases before reaching $y^\dag$, then decreases after exceeding $y^\dag$.

Then we aim to prove $\varphi(\nu^{\gamma_1-1}h^{(1-\alpha)\gamma_1-1}{\red)}\geq0$ and $\varphi(h^{(1-\alpha)\gamma_1-1})\geq0$.
Indeed, if $y = \nu^{\gamma_1-1}h^{(1-\alpha)\gamma_1-1}$, we obtain
$$
\begin{aligned}
&~~~~\varphi(\nu^{\gamma_1-1}h^{(1-\alpha)\gamma_1-1})\\
&= C_3(h)y^{r_1-\beta_1} + C_4(h)y^{r_2-\beta_1} + \frac{(\beta_1-1)h^{\alpha\beta_1}}{(r+\lam)(\beta_1-r_1)(\beta_1-r_2)} \\
&\geq C_3(h)y^{r_1-\beta_1} + (C_4(h)-C_6(h))y^{r_2-\beta_1} + \frac{(\beta_1-1)h^{\alpha\beta_1}}{(r+\lam)(\beta_1-r_1)(\beta_1-r_2)} \\
&= \frac{1-\beta_1}{(r+\lam)(r_1-r_2)(\beta_1-r_1)}h^{\alpha\beta_1}
+\frac{\beta_1-1}{(r+\lam)(r_1-r_2)(\beta_1-r_2)}\frac{h^{r_1\gamma_1+r_2+\alpha\beta_1}}{(\nu h)^{r_1\gamma_1+r_2}} \\
&~~+\frac{\beta_1-1}{(r+\lam)(\beta_1-r_1)(\beta_1-r_2)}h^{\alpha\beta_1} \\
&\geq \frac{(\beta_1-1)h^{\alpha\beta_1}}{r}\bigg( -\frac{1}{(r_1-r_2)(\beta_1-r_1)} + \frac{1}{(r_1-r_2)(\beta_1-r_2)} + \frac{1}{(\beta_1-r_1)(\beta_1-r_2)} \bigg)\\
&= 0,
\end{aligned}
$$
{\red where the last second inequality holds because $(\beta_1-r_2)(r_1\gamma_1+r_2)<0$ and $0 <\nu < 1$.}
On the other hand, if $y = h^{(1-\alpha)\gamma-1}$, we can obtain
$$
\begin{aligned}
&~~~~\varphi(h^{(1-\alpha)\gamma_1-1})\\
&\geq C_3(h)y^{r_1-\beta_1} + (C_4(h)-C_6(h))y^{r_2-\beta_1} + \frac{(\beta_1-1)h^{\alpha\beta_1}}{(r+\lam)(\beta_1-r_1)(\beta_1-r_2)} \\
&= \frac{1-\beta_1}{(r+\lam)(r_1-r_2)(\beta_1-r_1)}\frac{(\nu h)^{r_2\gamma_1+r_1}}{h^{r_2\gamma_1+r_1-\alpha\beta_1}}
+\frac{\beta_1-1}{(r+\lam)(r_1-r_2)(\beta_1-r_2)}h^{\alpha\beta_1} \\
&~~+\frac{\beta_1-1}{(r+\lam)(\beta_1-r_1)(\beta_1-r_2)}h^{\alpha\beta_1} \\
&\geq \frac{(\beta_1-1)h^{\alpha\beta_1}}{r+\lam}\bigg( -\frac{1}{(r_1-r_2)(\beta_1-r_1)} + \frac{1}{(r_1-r_2)(\beta_1-r_2)} + \frac{1}{(\beta_1-r_1)(\beta_1-r_2)} \bigg)\\
&= 0.
\end{aligned}
$$
\item[(iii)] In the region $y>(\nu h)^{\gamma-1}h^{\alpha\gamma}$,
similar to the proof of $C_5(h)\geq0$, we can get
$$
C_2(h) > C_2(h)-C_6(h) \geq 0.
$$
Therefore, $v_{yy}(y,h) = r_2(r_2-1)C_2(h)y^{r_2-2} + \frac{2\lam (\beta_2-1) K^{1-\beta_2}}{\kappa^2(\beta_2-r_1)(\beta_2-r_2)}y^{\beta_2-2}>0$.
\end{itemize}
\subsection{Proof of Corollary \ref{cor: asy_infty_wealth}}\label{proof: asy_infty_wealth}
Along the boundary $x_\mathrm{lavs}(h)$, we first have
$\frac{c^*(x_\mathrm{lavs}(h),h)}{x_\mathrm{lavs}(h)}
=  \frac{h}{x_\mathrm{lavs}(h)},$
where $x_\mathrm{lavs}(h)$ is defined in \eqref{eq: boundcurves_x}:
$$
\begin{aligned}
x_{\mathrm{lavs}}(h) &:= -C_5(h)r_1(1-\alpha)^{r_1-1}h^{-r_2((1-\alpha)\gamma_1-1)} - C_6(h)r_2(1-\alpha)^{r_2-1}h^{-r_1((1-\alpha)\gamma_1-1)}
\\ &~~~~-\frac{2\lam(1-\alpha)^{\beta_2-1}K^{1-\beta_2}}{\kappa^2(\beta_2-r_1)(\beta_2-r_2)}h^{(\beta_2-1)((1-\alpha)\gamma_1-1)} + \frac{h}{r+\lam} .
\end{aligned}
$$
Also, we have
$$
\begin{aligned}
C_5(h)r_1(1-\alpha)^{r_1-1}h^{-r_2((1-\alpha)\gamma_1-1)}
&= \frac{r_1(1-\alpha)^{-r_2}(\nu^{r_2\gamma_1+r_1}-1)(1-\beta_1)}{(r+\lam)(r_1-r_2)(\beta_1-r_1)}h, \\
C_6(h)r_2(1-\alpha)^{r_2-1}h^{-r_1((1-\alpha)\gamma_1-1)}
&= \frac{r_2(1-\alpha)^{-r_2}(1-\nu^{r_2\gamma_1+r_1})(1-\beta_1)(r_2(1-\alpha)\gamma_1+r_1)}{(r+\lam)(r_1-r_2)(\beta_1-r_1)(r_1(1-\alpha)\gamma_1+r_2)}h,\\
\end{aligned}
$$
and
$$
(\beta_2-1)((1-\alpha)\gamma_1-1) \leq 1,
$$
{\red thanks to \textbf{Assumption (A1)}}, and the equality holds if and only if $\gamma_2 = (1-\alpha)\gamma_1$.
Therefore, we have
$$
\begin{aligned}
\lim\limits_{h\rightarrow+\infty} \frac{x_\mathrm{lavs}(h)}{h}
&= - \frac{r_1(1-\alpha)^{-r_2}(\nu^{r_2\gamma_1+r_1}-1)(1-\beta_1)}{(r+\lam)(r_1-r_2)(\beta_1-r_1)} \\
&~~~~- \frac{r_2(1-\alpha)^{-r_2}(1-\nu^{r_2\gamma_1+r_1})(1-\beta_1)(r_2(1-\alpha)\gamma_1+r_1)}{(r+\lam)(r_1-r_2)(\beta_1-r_1)(r_1(1-\alpha)\gamma_1+r_2)}
\\
&~~~~- \frac{2\lam(1-\alpha)^{\beta_2-1}K^{1-\beta_2}}{\kappa^2(\beta_2-r_1)(\beta_2-r_2)}\mathbf{1}_{\gamma_2=(1-\alpha)\gamma_1} + \frac{1}{r+\lam}.
\end{aligned}
$$

The optimal investment on $x_\mathrm{lavs}(h)$ is
$$
\begin{aligned}
\pi^*(x_\mathrm{lavs}(h),h) &= \frac{2(r+\lam)}{\kappa^2}C_5(h)f_3(x_\mathrm{lavs}(h),h)^{r_1-1}
+  \frac{2(r+\lam)}{\kappa^2}C_6(h)f_3(x_\mathrm{lavs}(h),h)^{r_2-1}\\
&~~+ \frac{2\lam K^{1-\beta_2}(\beta_2-1)}{\kappa^2(\beta_2-r_1)(\beta_2-r_2)}f_3(x_\mathrm{lavs}(h),h)^{\beta_2-1} \\
&= \frac{2(r+\lam)(1-\alpha)^{-r_2}}{\kappa^2}C_5(h)h^{-r_2((1-\alpha)\gamma_1-1)} +  \frac{2(r+\lam)(1-\alpha)^{-r_1}}{\kappa^2}C_6(h)h^{-r_1((1-\alpha)\gamma_1-1)} \\
&~~~~+ \frac{2\lam K^{1-\beta_2}(\beta_2-1)(1-\alpha)^{\beta_2-1}}{\kappa^2(\beta_2-r_1)(\beta_2-r_2)}h^{(\beta_2-1)((1-\alpha)\gamma_1-1)}.
\end{aligned}
$$
Therefore, we conclude
$$
\lim\limits_{h\rightarrow+\infty} \frac{\pi^*(x_\mathrm{lavs}(h),h)}{x_\mathrm{lavs}(h)}
= \lim\limits_{h\rightarrow+\infty} \frac{\pi^*(x_\mathrm{lavs}(h),h)}{h}
\cdot \lim\limits_{h\rightarrow+\infty}  \frac{h}{x_\mathrm{lavs}(h)},
$$
which also exists.

The optimal bequest on $x_\mathrm{lavs}(h)$ is
$$
b^*(x_\mathrm{lavs}(h),h) = K^{-\frac{1}{\gamma_2-1}} \bigg((1-\alpha)h^{(1-\alpha)\gamma_1-1}\bigg)^{\frac{1}{\gamma_2-1}} = \bigg(\frac{1-\alpha}{K}\bigg)^{\frac{1}{\gamma_2-1}} h^{\frac{(1-\alpha)\gamma_1-1}{\gamma_2-1}}.
$$
Therefore, we conclude
$$
\begin{aligned}
\lim\limits_{h\rightarrow+\infty} \frac{b^*(x_\mathrm{lavs}(h),h)}{x_\mathrm{lavs}(h)}
&= \lim\limits_{h\rightarrow+\infty} \bigg(\frac{1-\alpha}{K}\bigg)^{\frac{1}{\gamma_2-1}} \frac{h}{x_\mathrm{lavs}(h)}\cdot h^{\frac{(1-\alpha)\gamma_1-\gamma_2}{\gamma_2-1}} \\
&= \mathbf{1}_{\gamma_2=(1-\alpha)\gamma_1}\bigg(\frac{1-\alpha}{K}\bigg)^{\frac{1}{\gamma_2-1}}\lim\limits_{h\rightarrow+\infty}  \frac{h}{x_\mathrm{lavs}(h)}
\end{aligned}
$$
is positive if $\gamma_2 = (1-\alpha)\gamma_1$, and equals 0 otherwise.
\ \\
\ \\
\textbf{Acknowledgements}: X. Yu is partially supported by PROCORE France/Hong Kong Joint Research Scheme under no.F-PolyU501/17 and the Hong Kong Polytechnic University research grant under no.P0031417. X. Li is partially supported by the Research Grants Council of Hong Kong under no. 15216720 and 15221621.
This work is also partially supported by the PolyU-SDU Joint Research Center on Financial Mathematics.

\bibliographystyle{plain}

\end{document}